\documentclass[11pt]{article}
\usepackage{amsmath, latexsym, amsfonts, amssymb, amsthm, amscd}
\usepackage{amsmath}
\usepackage{mathtools} 
\usepackage[left=2.6cm, right=2.6cm, top=3cm, bottom=3cm]{geometry}
\usepackage{upquote} 
\usepackage{algorithmicx}
\usepackage{color}
\usepackage{setspace} 
\usepackage[labelfont=bf]{caption} 
\usepackage{stmaryrd} 
\usepackage{multicol} 
\usepackage[dvipsnames]{xcolor}
\usepackage{lipsum}
\usepackage{kpfonts} 
\usepackage{dsfont} 
\usepackage{upgreek}
\usepackage{url}
\usepackage{float}

\usepackage[english]{babel} 
\usepackage[utf8]{inputenc}
\usepackage[T1]{fontenc}
\usepackage{enumerate}
\usepackage{graphicx}
\usepackage{color,soul}
\usepackage{comment}
\usepackage{caption}

\usepackage[toc,page]{appendix}
\usepackage[mathlines]{lineno}
\usepackage{etoolbox} 
\newcommand*\linenomathpatch[1]{%
	\cspreto{#1}{\linenomath}%
	\cspreto{#1*}{\linenomath}%
	\csappto{end#1}{\endlinenomath}%
	\csappto{end#1*}{\endlinenomath}%
}
\newcommand*\linenomathpatchAMS[1]{%
	\cspreto{#1}{\linenomathAMS}%
	\cspreto{#1*}{\linenomathAMS}%
	\csappto{end#1}{\endlinenomath}%
	\csappto{end#1*}{\endlinenomath}%
}
\expandafter\ifx\linenomath\linenomathWithnumbers
\let\linenomathAMS\linenomathWithnumbers
\patchcmd\linenomathAMS{\advance\postdisplaypenalty\linenopenalty}{}{}{}
\else
\let\linenomathAMS\linenomathNonumbers
\fi

\linenomathpatch{equation}
\linenomathpatchAMS{gather}
\linenomathpatchAMS{multline}
\linenomathpatchAMS{align}
\linenomathpatchAMS{alignat}
\linenomathpatchAMS{flalign}

\usepackage{bm} 

\usepackage{enumitem} 

\usepackage{xcolor}
\definecolor{Maroon}{HTML}{ad2231}
\definecolor{webgreen}{HTML}{008000}
\usepackage{dsfont}
\usepackage{hyperref}
\usepackage{lipsum}
\hypersetup{colorlinks, breaklinks, urlcolor=black, linkcolor=black, citecolor=webgreen} 
\allowdisplaybreaks
\usepackage{marginnote}

\usepackage{amsthm}
\newtheorem{theorem}{Theorem}

\newtheorem{proposition}[theorem]{Proposition}
\newtheorem{lemma}[theorem]{Lemma}
\newtheorem{remark}[theorem]{Remark}
\newtheorem{definition}[theorem]{Definition}

\theoremstyle{definition}

\usepackage{todonotes}
\newcommand{\Apostolos}[1]{\todo[color=cyan!20,size=\tiny]{Apostolos: #1}}

\usepackage{chngcntr}
\usepackage{apptools}
\AtAppendix{\counterwithin{lemma}{section}}
\newcommand{\email}[1]{\gdef\@email{\url{#1}}}

\newcommand{\R}{\mathbb{R}}

\newcommand{\E}{\mathbb{E}}

\DeclareMathAlphabet{\mathpzc}{OT1}{pzc}{m}{it}
\newcommand{\ee}{\textnormal{e}}
\newcommand{\dd}{\mathrm{d}}

\newcommand{\1}{\mathbf{1}}
\newcommand{\V}{\mathrm{V}}
\newcommand{\B}{\mathrm{B}} 
\newcommand{\p}{{p}} 
\newcommand{\f}{\mathrm{r}}  

\makeatletter
\def\widebreve{\mathpalette\wide@breve}
\def\wide@breve#1#2{\sbox\z@{$#1#2$}%
	\mathop{\vbox{\m@th\ialign{##\crcr
				\kern0.08em\brevefill#1{0.8\wd\z@}\crcr\noalign{\nointerlineskip}%
				$\hss#1#2\hss$\crcr}}}\limits}
\def\brevefill#1#2{$\m@th\sbox\tw@{$#1($}%
	\hss\resizebox{#2}{\wd\tw@}{\rotatebox[origin=c]{90}{\upshape(}}\hss$}
\makeatletter
\allowdisplaybreaks

\begin{document}
	\title{L\'evy processes with  partially stochastic resetting }
	\author{
 Zbigniew Palmowski\footnote{Department of Applied Mathematics, Wrocław Un. of Science and Technology, \texttt{zbigniew.palmowski@gmail.com}},	\, 	Noah Beelders\footnote{Department of Mathematical Sciences,
		University of Liverpool, \texttt{arkbeelder@gmail.com}}, \,  Lewis Ramsden\footnote{School for Businesses and Society, Univeristy of York, \texttt{lewis.ramsden@york.ac.uk}} \, \&  Apostolos D. Papaioannou\footnote{Department of Mathematical Sciences,
		University of Liverpool, \texttt{papaion@icloud.com}} } 
	
	\maketitle
	
	\tableofcontents
	
	\vspace{0.1in}
	
	\begin{abstract} 
	In this paper, we solve exit problems for a  L\'evy process that resets proportionally to its current position   at  independent  Poisson epochs times. This resetting causes an additional, proportional to its current  level,  downward (upward) jump when   the current position of the process is on the positive (negative)  domain.    All identities are given in terms of new family  of 
	scale function operators.  To obtain the new  scale function operators, we reduce the problem of the LT of the exit times into integral equations that are solve in terms of resolvent series. 
	\end{abstract}
	
	
	
	\section{Introduction}
In this paper, we study a spectrally negative L\'evy  process (SNLP) which experiences, at some (independent) Poisson epochs, a random jump whose size  is a random proportion of the state just before the jump.  Stochastic processes exhibiting this type of mechanism are referred to as \emph{processes with partial stochastic resetting} or  processes with \emph{additive-increase and multiplicative-decrease} (AIMD). In the special case where the process  resets to the origin (i.e.,~when the proportional component is absent), one obtains a process with \emph{total stochastic resetting}, a class of models originating in  \cite{EvS2011}. Due to their theoretical and practical significance,  such processes have  attracted a lot of attention and various results in relation to total resetting can be found in \cite{APZ2013, BFHMS2024, BKP2011, BLMP2021, BCGPST2025, DGR2022, GPST2026,  GPZ2004, AAM2016, HKPS2023, LLO2009}  as well as in the references of these papers, whilst the results for processes with partial resetting are limited, see for instance \cite{BCGPST2025,  HKPS2023}.   There is no doubt the aforementioned  processes have been motivated  from different fields.  Following the thorough and extensive discussion  in \cite{BKP2026, BCGPST2025}, the  total or partial   resetting process can be found in population genetics (see  for e.g. in   \cite{AEK2009}), in physics (see for e.g.  \cite{DCHPM2023}), in the transmission control protocol (TCP) networks (see  in \cite{DGP2002, GPZ2004, LL2008}), in cash management (see  \cite{BKP2026}), in stochastic thermodynamics (see   \cite{BCGPST2025}), in queuing (see \cite{AEL2007}), and in finance (see  \cite{YYL2005}).  

\vspace{1ex}

\noindent \textbf{Formulation of the problem}. To formulate our problem mathematically, we let  $X \equiv\left\{X_t\right\}_{t \geq 0}$ be a SNLP and   $N \equiv \{N_t\}_{t\geq 0}$  be  a Poisson process with   epochs $\left\{T_i\right\}_{i \in \mathbb{N}}$ of intensity $\lambda$.
The collapses are modelled by multiplying the present process position by a fixed proportion $p \in(0,1)$, i.e. $-\Delta X_{T_i}=(1-p) X_{T_i-}$ for $\Delta X_t=X_t-X_{t-}$. Therefore, we can define the \emph{partially stochastic resetting  L\'evy process} (PSR-LP), $U \equiv\left\{U_t\right\}_{t \geq 0}$,  starting at $U_0 = x_0$ with a resetting coefficient $p \in (0,1)$ as the solution to the stochastic differential equation (SDE)
\begin{equation}
	\dd U_t = \dd X_t -(1- p)U_{t^-} \dd N_t.  \label{Eq:SDEofPSR-LP}
\end{equation}
 It should now be apparent that  in addition to the negative jumps of the underlying process $X$, the process $U$ exhibits downward jumps whenever $U_{T_i^-}>0$, due to resetting, and upward jumps whenever $U_{T_i^-}<0$ (for the same reason).
It is also clear that the Brownian motion  component and the L\'evy jump components now provide mechanisms for $U$ to move above and below $0$, however, it is still not possible to move (spatially) between $\R_+$ to $\R_-$ (or vice-versa) by partial stochastic resetting. 
Clearly, when $p=1$, the process reduces to the dynamics of the underlying process $X$ and $p=0$ yields the total resetting SNLP. 
We note  that in the presence of partial stochastic resetting, the resulting trajectories tend to be closer to the origin, while they experience long excursions in the unperturbed case. This hints at the existence of a stationary state which was further investigated in  \cite{DCHPM2023}. A typical  sample path of the process $U$ is illustrated  in Figure \ref{figap1}.  
\vspace{-4ex}
\begin{figure}[H]
	\centering
	\includegraphics[width=0.6\textwidth,height=7cm]{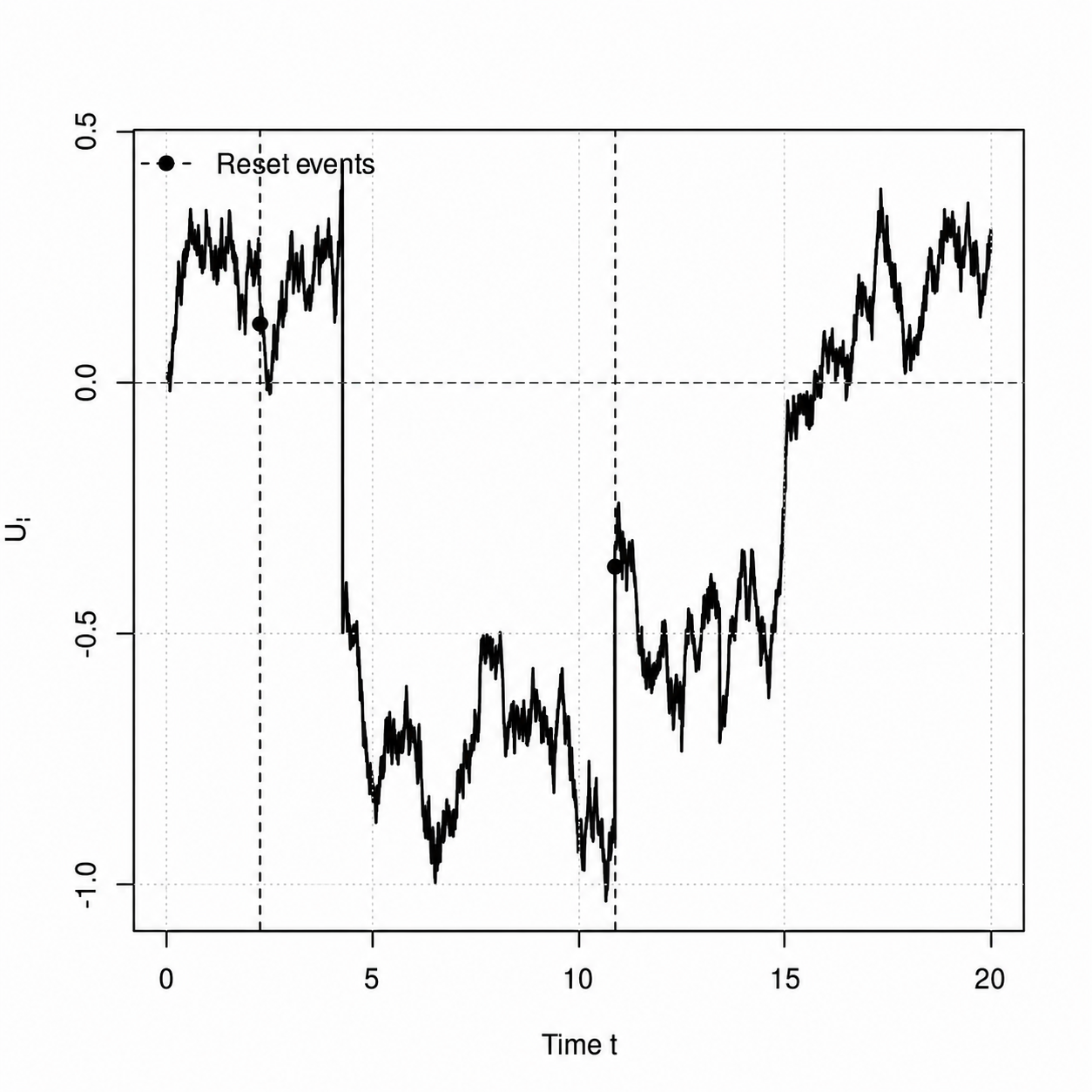}
	\caption{Spectrally negative L\'evy process with resetting with p=0.6}
	\label{figap1}
\end{figure}

\noindent \textbf{Contribution of the paper}. The main aim of this paper is to develop fluctuation theory for the PSR-LP, which is achieved by constructing a new class of scale  functions (based on potential measures rather than the classical scale functions) that fits the  PSR-LP set up.  Within this framework, the paper addresses two main difficulties arising for partial resetting process:  \textbf{(a)} Fluctuation theory of L\'evy processes typically relies on two key properties, namely  the Markov property and the skip-free property. However, as pointed out in \cite{HKPS2023} the key feature distinguishing AIMD processes from classical Lévy models is that the distribution of each jump depends on the position of $X$ prior to this jump. This dependence substantially complicates the study of associated exit identities and prevents the direct application of the standard machinery of Lévy fluctuation theory; for example, techniques based on Wiener–Hopf factorization  or Itô excursion theory.  To overcome these difficulties, we develop an alternative methodology based on potential densities  and integral equations. \textbf{(b)} As will be seen in the upcoming sections, PSR-LP exit identities are derived based on integral operators which evaluate functions at proportions (contractions) of their spatial points. This destroys the classical convolution or symmetric structure typically exploited in integral equations and complicates the direct application of standard spectral or Fourier techniques. To overcome this obstacle, we develop a solution based on the resolvent series. 

\noindent \textbf{Organization of the paper}. The paper is structured  as follows. In Section \ref{prelim},  we provide some basic theory of
scale functions and useful identities that will be used in the rest of the paper.  In Section \ref{Sec:PSR-LP Exit Identities},  the existence of the solution of the SDE in Eq.  \eqref{Eq:SDEofPSR-LP} is discussed and  we present our main
results (two and one-sided exit identities) with their proofs. Furthermore, we derive exit identities for the stochastic partial resetting process under a lower and also a upper reflection.  In the same section we show how these results can be reduced to the total stochastic resetting L\'evy process, by means of a limiting argument ($p\rightarrow 0$).


\color{black}	
	\section{Preliminaries}
	\label{prelim}
	Let $X \equiv \{X_{t}\}_{t \geq 0}$ be a SNLP defined on the filtered space $(\Omega, \mathcal{F}, \{\mathcal{F}_{t}\}_{t \geq 0}, \mathbb{P})$, where the filtration $\{\mathcal{F}_{t}\}_{t \geq 0}$ is assumed to satisfy the usual assumptions of right-continuity and completion. A L\'evy process with no positive jumps (the case of monotone paths is excluded) has a Laplace exponent $\psi(\vartheta):[0,\infty)\rightarrow \mathbb R$ such that
	\begin{equation*}
		\psi(\vartheta) := \log \E[e^{\vartheta X_1}], \;\;  \forall \vartheta\geqslant 0, 
	\end{equation*}
	for which the L\'evy-Khintchine formula shows that it has the form
	\[ \psi(\vartheta) = \mu\vartheta + \frac{\vartheta^2 \sigma^2}{2} + 
	\int_{(-\infty, 0)}^{} (e^{\vartheta x} -1- \vartheta x \mathbf{1}_{\{x >- 1\}}) \nu(\mathrm{d}x),\]
	where $\mu \in \R$, $\sigma \geq 0$ 
	and $\nu$  the L\'evy measure which is a $\sigma$-finite measure 
	concentrated on $(-\infty,0)$ satisfying $\int_{(-\infty,0)}(1 \wedge |x|^2) \nu(\mathrm{d}x)  <\infty$. 
	The above shows that $\psi$ is a continuous and strictly convex function, and that it tends to infinity as $\vartheta$ tends to infinity. Thus, for $q \geq 0$, one can define the right-inverse of the Laplace exponent $\Phi_q := \sup\{\vartheta: \psi(\vartheta) = q\},
	$ for which $\vartheta = 0$ is the unique solution to $\psi(\vartheta) = 0$ on $[0,\infty)$ if $\psi^{\prime}\left(0^{+}\right) \geq 0$ else there are two solutions. Further details about SNLPs can be found in the monographs of Bertoin \cite{B1996}, Kuznetsov et al.  \cite{KKR2012} and Kyprianou  \cite{K2014}.

	
	
	It is well-known that the fluctuation identities for $X$ relies heavily on the so-called $W^{(q)}$ and $Z^{(q)}$ scale functions (see Chapter 8 in Kyprianou \cite{K2014}). Hence, 
	for any $q\geq 0$, define $W^{(q)}: \R \to [0, \infty)$ and  $Z^{(q)}: \mathbb{R} \rightarrow[1, \infty)$  having the forms 
	\begin{equation}
		\int_{0}^{\infty} e^{-\vartheta x}W^{(q)}(x)\mathrm{d}x = \frac{1}{\psi_q(\vartheta)},\;  \vartheta > \Phi_q,  \quad \text{and}\quad  	Z^{(q)}(x)=1+q \int_0^x W^{(q)}(y) \mathrm{d} y,
		\label{eq:LTofscaleW}
	\end{equation}
	where $\psi_q(\vartheta) := \psi(\vartheta) -q$,  $W^{(q)}(x)=0$ for $x<0$. We note that  $W^{(q)}(x)$ is right  continuous increasing function and that $Z^{(q)}$ inherits the properties of $W^{(q)}$. In the
	rest of the paper, when $q= 0$, we shall drop the subscript.
	
	By considering the Laplace transform of $W^{(q)}$, one can deduce that $W^{(q)}(0+)=1 / c$ when $X$ is of bounded variation since it necessarily has the form $c t-S_t$, where $S=\left\{S_t: t \geq 0\right\}$ is a driftless subordinator and $c>0$. In contrast, the case when $X$ is of unbounded variation yields $W^{(q)}(0+)=0$. For more information   texts such as Chan et al.~\cite{CK2011}, Kuznetsov et al.~\cite{KKR2012}, and Kyprianou et al. \cite{KRS2010}  can be consulted for further discussions on the intricacies of the smoothness of scale functions as well as other facts about them. 
	
	\noindent 	With regards to the limits of scale functions, it is well-known that 
		\begin{equation}
			\lim_{a \rightarrow \infty} \frac{W^{(q)}(a+x)}{W^{(q)}(a)} = e^{\Phi_q x}, \; \quad \lim_{a \rightarrow \infty} \frac{Z^{(q)}(a)}{W^{(q)}(a)} = \frac{q}{\Phi_q}. \label{eq:LimitofRatioofScaleFuncs}
		\end{equation}
Potential measures that are known to play a fundamental role in fluctuation theory of L\'evy processes and   	are strongly  connected with the scale functions (in the case of SNLPs),  will be (mainly) used throughout this paper. Therefore, for 
\[
\tau_{a,X}^{+(-)}:=\inf \left\{t>0: X_t>(<)\; a\right\},\]
where $X$ in their subscripts indicate the underlying process that is considered and for 
any $a>0$  , $x, y \in[0, a]$, $q \geq 0$ the q-potential measure of $X$ killed on exiting $[0,a]$, with density (with respect to Lebesgue measure) $	\mathrm R^{(q)}(x,\dd y)\equiv\f^{(q)}(x,y)\dd y$,   given by  
	\begin{equation}
	\int_0^{\infty} \mathrm{e}^{-q t} \mathbb{P}_x\left(X_t \in \mathrm{d} y, t<\tau_a^{+} \wedge \tau_0^{-}\right) \mathrm{d} t= \bigl[\frac{W^{(q)}(x) }{W^{(q)}(a)}W^{(q)}(a-y)-W^{(q)}(x-y)\bigr] \mathrm{d} y.   \label{eq:ClassicalKilledPotential}
	\end{equation}
	In fact,  we shall often make use of an important fact about absolutely continuous potential measures given in the following proposition. 
	\begin{proposition}[Proposition  10 in Bertoin  \cite{B1996}, pp.~25] \label{Prop:BertoinProp}
		The following assertions are equivalent:
		\begin{itemize}
			\item[\upshape {(i)}] For every $q > 0$ and every $x \in \R$, the potential measure $\mathrm R^{(q)}(x,\dd y)$ is absolutely continuous w.r.t.~the Lebesgue measure.
			
			\item[\upshape{(ii)}] The resolvent operator 
			\[
		\mathrm 	R^{(q)} f(x) := \int_{\R} f(y) \mathrm R^{(q)}(x, \dd y)
			\] 
			has the strong Feller property; that is, for every $q > 0$ and $f \in L^\infty(\R)$, the function $\mathrm R^{(q)} f$ is continuous, where $L^\infty(\R)$ denotes the space of essentially bounded measurable functions with domain $\R$ w.r.t.~the Lebesgue measure.
		\end{itemize}
	\end{proposition}

\noindent Finally, the fluctuation identities of $X$, which are used throughout the paper are given in Chapter 8 in  \cite{K2014}. In more details  for  $q \geq 0$ and $x \leq a$, it holds that 
	\begin{equation}
		\mathbb{E}_x\left(\mathrm{e}^{-q \tau_{a,X}^{+}} \mathbf{1}_{\left\{\tau_{a,X}^{+}<\tau_{0,X}^{-}\right\}}\right)=\frac{W^{(q)}(x)}{W^{(q)}(a)},  \quad 	\mathbb{E}_x\left(\mathrm{e}^{-q \tau_{0,X}^{-}} \mathbf{1}_{\left\{\tau_{0,X}^{-}<\tau_{a,X}^{+}\right\}}\right)=Z^{(q)}(x) - \frac{W^{(q)}(x)}{W^{(q)}(a)}Z^{(q)}(a) . \label{eq:ClassicalExitfromBelow} 
	\end{equation}

	\section{Main results}\label{Sec:PSR-LP Exit Identities}

In this section we derive the two-sided exit upward and downward identities and their corresponding 
	   one-sided  exit upward and downward  identities.   It will be shown that these  exit times are provided by means of a new  family of scale functions.  It is worth mentioning that due to  the behaviour of the partial resetting in $\R_-$  (resetting occurs in an upwards directions), the PSR-LP is no longer spectrally negative when it enters $\R_-$. It is therefore anticipated that a fractional form of the two-sided exit upwards for the PSR-LP in $\R_-$ will not be established as occurs in the classical case.  Before we proceed with  the fluctuation theory results, we show that $U$ satisfying  the SDE in Eq. \eqref{Eq:SDEofPSR-LP} exists and  is unique.  Let us note that existence and uniqueness of this SDE follows
	   from the fact that process $U$ is a special case of the SDE describing a Generalised Ornstein-Uhlenbeck
	   process; see \cite{B2011} for a monograph on such processes and \cite{K2019} for the case when $X$ is a stable process. Nevertheless, in the present setting the proof is considerably more elementary and  we therefore include it not only for completeness, but also because the same argument extends to Skorokhod-reflected processes; see Subsection~\ref{reflection}. 
	 	\begin{theorem}\label{sdepart}
	 There exists unique strong solution  $U$ of the SDE in Eq~\eqref{Eq:SDEofPSR-LP}. Also, it is a strong Markov process.
	 \end{theorem}
	 
	 \begin{proof}
	 Recall that  $
	 N_t := \sum_{k \geqslant 1} \1_{\{T_k \leq t\} }$, 	with $0=T_0<T_1<T_2<\cdots<T_n$,  such that $T_n\rightarrow \infty$, as $n\rightarrow \infty$. 
	 	Integrating Eq.~\eqref{Eq:SDEofPSR-LP}  over the set $\left\{T_k\right\}_{k \geqslant 0}$ gives	 	
\[	 		 	U_{T_k}=\bigl(X_{T_k}-X_{T^-_{k}}\bigr)+pU_{T^-_{k}}. 
\]
Hence, $U_{T_k}$  is determined by $U_{T^-_k}$.  On the other hand, integrating  Eq. \eqref{Eq:SDEofPSR-LP} over the open interval $[T_{k-1}, T_k)$ and noticing that the integral term w.r.t. $N_t$ is 0,  gives
	 	\[
	 	U_{T^-_{k}}-U_{T_{k-1}}=X_{T^-_{k}}-X_{T_{k-1}}
	 	\]
	 	which shows  that 	$U_{T^-_{k}}$ is determined  by   $U_{T_{k-1}}$ (which itself is determined by $U_{T_{k-1}^-}$). Hence, by induction $U_{T_k}$ is determined for all $ k\geq 0$.
	 	Now, for any  $t>0, t \neq T_k$,  there is a unique $k$ s.t.~$T_k<t<T_{k+1}$ for which integrating Eq.~\eqref{Eq:SDEofPSR-LP} over $\left(T_k, t\right]$ gives
	 	\[
	 	U_t-U_{T_k}=X_t-X_{T_k}. 
	 	\]
	 Hence, $U_t$ is determined by $U_{T_k}$ and therefore $U$ exist for all $t$. To show the uniqueness, let $U_t^i$, $i=1,2$, be two solutions of Eq. \eqref{Eq:SDEofPSR-LP}  with the same initial condition and define $D_t:=U_t^1-U_t^2$. Then,  from  Eq.~\eqref{Eq:SDEofPSR-LP}, it holds that 
	 \[\dd D_t=-(1-p)D_{t^-}\dd N_t, \quad D_0=0, \]
	 which has an  explicit  solution of the form $D_t=p^{N_t}D_0$ for any $t\geq 0$ (with $p^{N_t}$  the Dol\'eans-Dade exponent) and therefore the solution is unique. The strong Markov property follows from Section 3 of \cite{GPST2026}.
	 \end{proof}
 \subsection{A universal partial resetting integral equation}
 \label{sec3.1}
We note that the fluctuation identities in the subsequent subsections will be provided by means of solutions of integral equations. Therefore, we shall now derive a  general solution of   a  (general) integral equation that will be used as a universal tool in all the subsequent exit problems. This will be achieved by using  the    the Banach fixed point theorem in the spirit of Chapter 2 in Hochstadt \cite{H1989} or in Kolmorogov and Fomin  \cite{KF1975}, Section 8 pp.~66-77.

Consider $(C_\mathrm{b}(\B), || \cdot ||)$  a Banach  space of bounded continuous functions with domain $\B$ in $\mathbb R$
where $|| \cdot || := \sup_{x \in \B} | \cdot |$ is the standard supremum norm.  Without any loss of generality, the derivation of the exit identities only requires three types of intervals in $\R$. Thus we enforce the following assumption on set $\B$ for the rest of the paper.
	
	\medskip
	
	\noindent \textbf{(H1):} For $a,b \in \R$ s.t.~$b \leq a$, we let $\B = [b, a]$, $(-\infty, a]$ or $[b, \infty)$.
	
	\medskip
	
	\noindent In the classical theory of integral equations, it is well known that the resolvent series (which sometimes  called the Neumann series) as well as its associated parameters are completely specified and determined by a recursion of some starting kernel.  Given that each of the exit problems are associated with different starting kernels, we shall make use of the following definition to ensure that the identities in the remainder of the paper are formulated concisely.
	\begin{definition}
		\label{Def:ResolventDefn}
	Let $p \in [0,1]$ and $q >0$.  We fix $[u_1,u_2] \subset \B$ Then, for  $\gamma \in \R$ and $x, \, y \in \B$, we define the resolvent series
		\begin{equation}
			\overline{\V}^{(p)}_\gamma [\f^{(q)}](x,y) :=  \sum^{\infty}_{k=1} \gamma^k \V ^{(p)}_{k}[ \f^{(q)}](x,y) \label{Eq:ResolventNeumannSeries}
		\end{equation}
in terms of the recursion
		\begin{align}
			\V_{1}^{(p)}[\f^{(q)}](x,y) &:=\f^{(q)}(x,y), \notag \\
			\V_{n}^{(p)}[\f^{(q)}](x,y) & : = \int^{u_2}_{u_1} \f^{(q)}(x,z) \, \V_{n-1}^{(p)}[\f^{(q)}](pz,y)\, \dd z, \quad n \geq 2, \label{Eq:ResolventforScaleFuncs}
		\end{align}
		where $x,y \in  [u_1,u_2]\subset \B$,  and $\f^{(q)}$ denotes the particular starting kernel.
	\end{definition}
	\noindent As has been alluded to, the starting kernels $r^{(q)}$ chosen for our purposes will always have the special property of being a density of a potential measure.
\begin{remark}
	\label{ap4}\upshape{
We point out that the usage of the potential density $\f^{(q)}$  in the resolvent series from Eq.~\eqref{Eq:ResolventNeumannSeries} instead of the underlying scale functions $W^{(q)}$ is mathematically necessary. The reason for this, as will be seen in the proof of  Theorem \ref{Thm:Res-UAuxiliaryFuncIsFixedPoint} and Remark \ref{Rem:WellDefinedIntegralOperators} below, is that the scale functions don't  satisfy the conditions of the Banach fixed point theorem which is used to solve the individual integral equations whereas the potential measures succeed due to their useful boundedness properties.
}
	\end{remark}
	\noindent We shall now state the universal integral partial resetting equation that will be used in the subsequent subsections to solve the exit problems.
	\begin{theorem}\label{Thm:Res-UAuxiliaryFuncIsFixedPoint}
	 Let $p \in (0,1)$, $q > 0$ and $\gamma \in \R$ s.t.~$0<|\gamma| < q$, and $[u_1,u_2] \subset \B$ s.t.~$[pu_1,pu_2] \subset \B$. For $f$ and $h \in  C_\mathrm{b}(\B)$, consider the integral operator $\mathcal{A} : C_\mathrm{b}(\B) \rightarrow C_\mathrm{b}(\B)$ s.t.~
		\begin{equation}
			\mathcal{A}f(x) = h(x) + \gamma \int^{u_2}_{u_1} \f^{(q)}(x,y) f(py) \dd y.   \label{Eq:Res-MainIntegralOperator}
		\end{equation}
		Then, the (unique) fixed point $g \in C_\mathrm{b}(\B)$ of $\mathcal{A}$ is of the form
		\begin{equation}
			g(x) = h(x) + \int^{u_2}_{u_1} \overline{\V}^{(p)}_\gamma[\f^{(q)}](x,y)  \, h(py)\, \dd y,\label{solop}
		\end{equation} 
		where $\overline{\V}^{(p)}_\gamma[r^{(q)}]$ is given in Definition \ref{Def:ResolventDefn}. 
		%
	\end{theorem}
	
	\begin{remark}\label{Rem:WellDefinedIntegralOperators}
		\upshape{
	\noindent 
	 The image of the mapping of $\mathcal{A}$ in Theorem \ref{Thm:Res-UAuxiliaryFuncIsFixedPoint} is a consequence of the potential density and  that the constituent functions $f,h \in C_\mathrm{b}(\B)$. Indeed, for $q >0$, since $\f^{(q)}$ is a density of a potential measure of an killed SNLP, we have for $[u_1,u_2] \subset \B$ that  
		\begin{align}
			\int^{u_2}_{u_1} \f^{(q)}(x,y) \dd y &\leq \E_x\Bigl(\int^\infty_0 \ee^{-qt} \1_{\{X_t \in [u_1,u_2]\}} \, \dd t \Bigr) \leq \E_x\Bigl(\int^\infty_0 \ee^{-qt} \, \dd t \Bigr) = \frac{1}{q}, \label{Eq:ResolventInequality}
		\end{align}
		and hence, by the triangle inequality, that $
		|| \mathcal{A} f(x) || \leq || h(x) || + \frac{|\gamma|}{q} || f(x) || < \infty, $
		since $f$ and $h$ are bounded. The continuity of $\mathcal{A} f(x)$,  follows from the strong Feller property in Proposition \ref{Prop:BertoinProp}, and the continuity of $f$ and $h$, and thus $\mathcal{A} f(x) \in C_\mathrm{b}(\B)$}
	\end{remark}
\begin{remark}
	\label{rem2}\upshape{
		Observe from Definition \ref{Def:ResolventDefn} that since $h \in C_b(\B)$, for some $M \geq 0$, $k \geq 2$ and $y \in [u_1,u_2]$, we have 
		\[\Bigl| \Bigl| \gamma^k\int_{u_1}^{u_2}\V_{n}^{(p)}[\f^{(q)}](x,y)h(py)\dd y\Bigr| \Bigr| \leq |\gamma|^k\frac{M}{q}\Bigl| \Bigl|\int^{u_2}_{u_1}  \V_{n-1}^{(p)}[\f^{(q)}](pz,y)\, \dd y  \Bigl| \Bigl| \leq \Bigl(\frac{|\gamma|}{q}\Bigr)^kM,\]
	which is independent of $p$. Therefore, the integral term in Eq. \eqref{solop} converges uniformly in  $C_\mathrm{b}(\B)$. Also note that for $h\equiv 1$, the above guarantees that  the series in Eq. \eqref{Eq:ResolventNeumannSeries} converges in $C_b(\B)$.}
\end{remark}
	
	\begin{proof}[Proof of Theorem \ref{Thm:Res-UAuxiliaryFuncIsFixedPoint}]
		Let us denote $\mathcal{A}^0f(x) = f(x)$ and  $\mathcal{A}^{n+1}f(x) = \mathcal{A}[\mathcal{A}^{n}f(x)]$ for $n \geq 0$. Following similar arguments as in  Hochstadt \cite{H1989}, by using the inequality in Eq.~\eqref{Eq:ResolventInequality} and that $[pu_1,pu_2] \subset \B$, we have for $f_1,f_2 \in C_\mathrm{b}(\B)$ that
		\begin{align}
			\bigl| \mathcal{A} f_1(x) - \mathcal{A} f_2(x) \bigr| 
			&\leq  |\gamma| \int^{u_2}_{u_1} \f^{(q)}(x,y) \, \bigl| f_1(py) - f_2(py)  \bigr|  \dd y 
			\leq \; \frac{|\gamma|}{q} \, \sup_{x \in \B} | f_1(x) - f_2(x)|. \notag
		\end{align}
		Thus, under the assumption $|\gamma|<q$, we obtain from the above inequality that
		\begin{align}
			\bigl| \bigl| \mathcal{A}f_1(x) - \mathcal{A}f_2(x) \bigr|  \bigr| \leq \frac{|\gamma|}{q}  \; || f_1(x) - f_2(x) ||,  \label{Eq:Res-ConvergenceOfIntegralOperator}
		\end{align}
		from which we conclude, for any $p \in (0,1)$, that $\mathcal{A} f(x)$ is a contraction mapping. Hence, by Theorem 1 pp.~26 in Hochstadt \cite{H1989}, there exists a unique fixed point of $\mathcal{A}$ that can be derived by taking $\lim_{n \rightarrow \infty} \mathcal{A}^nf$ for any choice of starting function $f \in C_\mathrm{b}(\B)$.
		Indeed, by using Eqs.~\eqref{Eq:ResolventInequality} and \eqref{Eq:Res-ConvergenceOfIntegralOperator}, we obtain inductively, for $n \geq 0$, that
		\begin{align} 
			\bigl| \bigl| \mathcal{A}^{n+1}f_1(x) - \mathcal{A}^{n+1}f_2(x) \bigr|  \bigr| &\leq  \Bigl| \Bigl| \gamma \int^{u_2}_{u_1} \f^{(q)}(x,y) \, \bigl| \mathcal{A}^{n}f_1(py) - \mathcal{A}^nf_2(py) \bigr| \, \dd y \Bigr| \Bigr| \notag \\
			&\leq  \Bigl| \Bigl| \gamma \int^{u_2}_{u_1} \f^{(q)}(x,y) \, \dd y \Bigr| \Bigr| \cdot \, \bigl| \bigl| \mathcal{A}^{n}f_1(x) - \mathcal{A}^nf_2(x) \bigr|  \bigr| \notag \\
			&\leq \Bigl(\frac{|\gamma|}{q} \Bigr)^{n+1} \;  || f_1(x) - f_2(x) ||,  \label{Eq:ConvergenceOfIntegralOperatorForalln}
		\end{align}
		and hence that $\lim_{n \rightarrow \infty} \mathcal{A}^nf_1 = \lim_{n \rightarrow \infty} \mathcal{A}^nf_2$ exists. 
		
	\noindent 	We now derive this unique fixed point iteratively. Let the integral operator
		\begin{equation}
			\mathcal{V}f(x) := \int^{u_2}_{u_1} \f^{(q)}(x,y) f(py) \dd y, \notag 
		\end{equation}
		such that  $	\mathcal{A}f(x) = h(x) + \gamma \mathcal{V}f(x)$ and iteratively  
		\begin{align}
			\mathcal{A}^n f(x) &= h(x) + \sum^{\infty}_{k=1} \gamma^k \mathcal{V}^kh(x), \notag
		\end{align}
		with  $\gamma^{n} \mathcal{V}^{n}f(x) \rightarrow 0$ as $n \rightarrow \infty$. Therefore, taking $n \rightarrow \infty$ of the above equation yields that the unique fixed point of $\mathcal{A}$ has the form
		\begin{align}
			g(x) = \lim\limits_{n\rightarrow \infty}\mathcal{A}^nf(x) =  h(x) + \sum^{\infty}_{k=1} \gamma^k \mathcal{V}^kh(x). \label{Eq:Res-FinalSolutionThm1} 
		\end{align}
	Now, by induction,    the integral operator $\mathcal{V}^n$ can be expressed in terms of the recursive functions $\V^{(p)}_n[r^{(q)}](x,y)$ in  Eq.~\eqref{Eq:ResolventforScaleFuncs}, i.e. that 
		\begin{equation}
			\mathcal{V}^nh(x) = \int^{u_2}_{u_1} \V^{(p)}_n[\f^{(q)}](x,y)\,h(py)\, \dd y.
		\end{equation}
		 Therefore, by substituting the above back into Eq.~\eqref{Eq:Res-FinalSolutionThm1} and using Definition \ref{Def:ResolventDefn}  and the fact that the below series is absolutely convergent 
the result follows. 
	 \end{proof}
		\begin{remark}\label{valuesofr}
	\upshape{The integral operator introduced in Theorem~\ref{Thm:Res-UAuxiliaryFuncIsFixedPoint}, together with its unique fixed point given in \eqref{solop}, will be used repeatedly throughout the following subsections. Indeed, the fluctuation identities derived therein will be shown to satisfy integral equations that fall within  Theorem~\ref{Thm:Res-UAuxiliaryFuncIsFixedPoint}. The specific form of the potential density $\f^ 
		{(q)}$, however, depends on the exit problem under consideration (e.g., one- or two-sided, upward or downward exit), since the corresponding scale function representations differ in each case. Accordingly, in the subsequent subsections, the generic kernel $\f^ 
		{(q)}$
		will be replaced by the appropriate potential density, expressed explicitly in terms of scale functions, as presented below.
		
		\medskip
		
			\noindent
				\noindent \textbf{(S1):} For $a,b \in \R$ s.t.~$b \leq a$, we let $\f^{(q)}(x,y)\equiv	r^{(q)}(x,y)$, 	$\overline{	r}^{(q)}(x,y)$, $\underline{r}^{(q)}(x,y) $, where 
			\begin{align}
				r^{(q)}(x,y) &:= \frac{W^{(q)}(x-b)}{W^{(q)}(a-b)}W^{(q)}(a-y) - W^{(q)}(x-y), \notag \\
				\overline{	r}^{(q)}(x,y)	 &:= \ee^{-\Phi_q (a-x)}W^{(q)}(a-y) - W^{(q)}(x-y), \notag \\
				\underline{r}^{(q)}(x,y) &:= \ee^{-\Phi_q (y-b)} W^{(q)}(x-b)  - W^{(q)}(x-y). \notag
			\end{align}	 
			}
	\end{remark}
	\subsection{Exit identities }\label{Subsubsec:Res-ExitUpwardsPostiveHalf}
	In this section, we shall derive an explicit identity for the two and the one sided exit  problems. We note that due to the behaviour of creeping versus jumping, the derivation of the two exit times encompasses different properties requiring typically  different techniques. Here, we propose a unifying framework for exit times of both one- and two-sided exit problems. In the rest of the paper $\mathbb E_x$ denotes the expectation with respect to $\mathbb P_x$. All fluctuation identities derived below are based on a new family of scale functions operators,  $\mathcal{W}_{b,a}^\p$,   $\mathcal{Z}_{b,a}^\p:\mathcal  R_{b,a }\to C_{\mathrm{b}}([b,a])	$, with $\mathcal R_{a,b}$  the class of admissible killed-potential densities appearing in \textbf{(S1)}, be the  operators of the  forms, for $a$,  $b$,  $x\in \mathbb R$ s.t.~$a>b$, 
		\begin{align}
				\begin{split}
	{\mathcal{W}_{b,a}^\p}\f^{(q)}(x)&= {W}^{(q+\lambda)}(x-b)+\int_{b\vee(a\wedge bp^{-1})}^{a\wedge(ap^{-1}\vee b)}\overline{\V}^{(p)}_\lambda[\f^{(q)}](x,y)W^{(q+\lambda )}(py -b)\dd  y, \\
		{\mathcal{Z}_{b,a}^\p}\f^{(q)}(x)&= {Z}^{(q+\lambda)}(x-b)+\int_{b\vee(a\wedge bp^{-1})}^{a\wedge(ap^{-1}\vee b)}\overline{\V}^{(p)}_\lambda[\f^{(q)}](x,y)Z^{(q+\lambda )}(py -b)\dd  y, 
			\label{Eq:PSR-ScaleFunc}
		\end{split}
		\end{align} 
	where $x\wedge(\vee) y=\text{\upshape {min}(max)}(x,y)$ and   $\overline{\V}^{(p)}_\lambda[\f^{(q+\lambda)}]$ as in  Definition \ref{Def:ResolventDefn}, with starting kernel the potential density $\f^{(q)}$,  defines in Section \ref{prelim}. We note that in a similar manner as in Remark  \ref{valuesofr}, depending on the exit problem under consideration $\f^{(q)}$ will take the forms  given in  \textbf{(S1)}. We now state our main results, given in Theorems \ref{Thm:Res-TwoSidedExitUpwards-PositiveHalfAxis} and \ref{Thm:Res-TwoSidedExitDownwards-PositiveHalfAxis}, which study upward and downward  two-sided exit  crossing problems, and Propositions \ref{Thm:Res-OneSidedExitDownwards-PositiveHalfAxis} and \ref{Thm:Res-OneSidedExitUpwards-PositiveHalfAxis}, which study one-sided exit problems. We start by discussing the upward two-sided exit problem. All of these exit problems are related to the following exit times 
	\[
	\tau_{a,U}^{+(-)}:=\inf \left\{t>0: U_t>(<)\; a\right\}.\]
	 We note that in the following results we consider $p \in (0,1)$ in order to avoid, for now,  the discussion of the total resetting case. 
	\begin{theorem}[Upward two-sided exit problem]\label{Thm:Res-TwoSidedExitUpwards-PositiveHalfAxis}
	Let $p \in (0,1)$,  $q >0 $ and $\lambda \geq 0$. Then,  $x \in [b,a]$ and $a>b \in \R$  for the following identity holds 
	\begin{align}
	\E_x \Bigl( \ee^{-q \tau_{a,U}^+} \1_{\{\tau_{a,U}^+ < \tau_{b,U}^-\}} \Bigr)= \frac{{\mathcal{W}}_{b,a}^{\p}r^{(q+\lambda)}(x)}{{\mathcal{W}}_{b,a}^{\p}r^{(q+\lambda)}(a)}+ \int_{ap^{-1}\vee b}^{a}\overline{\V}^{(p)}_\lambda[r^{(q+\lambda)}](x,y) \dd y \1_{\{a, b\in \R_-\}}, 
	\label{Eq:PSR-ExitUpwards} 
	\end{align} 
where 
$r^{(q)}$   given in   \textbf{\upshape {(S1)}}.
	\end{theorem}
	
	\begin{remark}\upshape{
		Notice that the above identity in the case that $a, \, b\in \R_-$ is not in a fractional form as in the cases of $a, \, b \in \R_+$ and $a\in \R_+$, $b\in \R_-$. This is largely expected since the exit upwards in the negative domain can occur by a reset towards zero which essentially yields an upward jump. 
	Hence, 	the lack of creeping is thus the main reason why the case $a,\, b \in \R_-$  does not have a fractional form. Additionally, from  Eq.~\eqref{Eq:PSR-ExitUpwards}, for $a,\, b \in \R_-$, it should be clear since $X$ is a SNLP that the exit identity of $U$ is decomposed into a term where crossing can occur by creeping (the fraction term) and the (integral) term which accounts for crossing by resetting.   }
	\end{remark}

	\begin{proof}[Proof of Theorem \ref{Thm:Res-TwoSidedExitUpwards-PositiveHalfAxis}]
		Let $a,b \in \R$ and  $g_{a,b}^{(q)}(x) := 	\E_x \bigl( \ee^{-q \tau_{a,U}^+} \1_{\{\tau_{a,U}^+ < \tau_{b,U}^-\}} \bigr) $. Then, denoting by $T \overset{\dd}{=} T_1$ the next Poisson epoch,  conditioning on $T$ and noticing that $\{U_t : t < T\} \overset{\dd }{=} \{X_t : t < T\}$, we  get
		\begin{align}
			g_{a,b}^{(q)}(x) 
			&= \E_x \Bigl( \ee^{-(q+\lambda) \tau_{a,X}^+} \1_{\{\tau_{a,X}^+ < \tau_{b,X}^-\}} \Bigr) + \E_x \Bigl( \ee^{-q T} \1_{\{T < \tau_{a,X}^+ \wedge \tau_{b,X}^- \}} g_{a,b}^{(q)}(pX_T) \Bigr) \notag \\
			&= \frac{W^{(q+\lambda)}(x-b)}{W^{(q+\lambda)}(a-b)} + \lambda \int^a_{b} \E_x \Bigl( \int^\infty_0 \ee^{-(q+\lambda)t} \1_{\{X_t \in \dd y, \, t < \tau_{a,X}^+ \wedge \tau_{b,X}^- \}} \dd t \Bigr) g_{a,b}^{(q)}(py) \notag \\
			&= \frac{W^{(q+\lambda)}(x-b)}{W^{(q+\lambda)}(a-b)} + \lambda \int^a_{b}  r^{(q+\lambda)}(x,y) \, g_{a,b}^{(q)}(py) \dd y, \label{Eq:PSR-ExitUpwardsMainRenewal} 
		\end{align}
		where we have used Eq.~\ \eqref{eq:ClassicalKilledPotential},  along with \textbf{(S1)}. Now, we shall use the above integral (renewal-type) equation to solve the Eq. \eqref{Eq:PSR-ExitUpwardsMainRenewal}.
		\vspace{0.2cm}\\
		(i) Consider $a,b \in \R_+$. Noticing for $a > bp^{-1}$ that $g_{a,b}^{(q)}(py) = 0$,  $y<bp^{-1}$ and for  $bp^{-1} \geq a >b$  that $g_{a,b}^{(q)}(py) = 0$ for all $y \in [b,bp^{-1})$, we get that  Eq. \eqref{Eq:PSR-ExitUpwardsMainRenewal} is reduced to 
			\begin{equation}
			g^{(q)}_{a,b}(x) = \frac{{W}^{(q+\lambda)}(x-b)}{{W}^{(q+\lambda)}(a-b)}  + \lambda \int^a_{a \wedge bp^{-1} } r^{(q+\lambda)}(x,y) \, g_{a,b}^{(q)}(py) \dd y. \label{Eq:Res-ResolventIntegralEqFinal}
		\end{equation}
		Now, observe that $g_{a,b} \in C_\mathrm{b}([b,a])$ since $g_{a,b}^{(q)}(x) \leq 1$, and that continuity follows by using Eq.~\eqref{Eq:Res-ResolventIntegralEqFinal},  and the strong Feller property from Proposition \ref{Prop:BertoinProp}.
		Further,  from Eq.~\eqref{Eq:Res-ResolventIntegralEqFinal} it holds that   $g_{a,b}^{(q)}(x) = \mathcal{A}g_{a.b}^{(q)}(x)$ is the fixed point of the above integral operator, and thus from Theorem \ref{Thm:Res-UAuxiliaryFuncIsFixedPoint}, we get that 
			\begin{equation*}
			g_{a,b}(x)	= \bigl[ {W}^{(q+\lambda)}(a-b) \bigr]^{-1}\Bigl[{W}^{(q+\lambda)}(x-b)+ \int^a_{a\wedge bp^{-1}} \overline{\V}^{(p)}_\lambda[r^{(q+\lambda)}](x,y) {W}^{(q+\lambda)}(py-b) \dd y \Bigr], \label{Eq:Res-AlmostFinalSolutionThm1} 
		\end{equation*}
		from which, by noticing that $\overline{\V}^{(p)}_\lambda[r^{(q+\lambda)}](a,y)=0$, the results follows immediately. 
		\vspace{0.2cm}\\
		\noindent 	(ii) Consider $a\in \R_+$, $b\in \R_-$. This case is exactly similar to (i) with the only difference that the integral limits are from $a$ to $b$. 
		\vspace{0.2cm}\\
					\noindent 	(iii)  We now consider $a,b \in \R_-$. Noticing  for $ap^{-1} > b$, that  $g_{a,b}^{(q)}(py) = 1$ for $y \in [ap^{-1},a]$ and for  $ap^{-1} \leq b < a$ that  $g_{a,b}^{(q)}(py) = 1$ for all $y \in [b,a)$,  it yields  that  Eq. \eqref{Eq:PSR-ExitUpwardsMainRenewal} can be written as 
			\begin{align}
			g^{(q)}_{a,b}(x) &= \frac{W^{(q+\lambda)}(x-b)}{W^{(q+\lambda)}(a-b)} + \lambda \int^a_{ap^{-1} \vee b} r^{(q+\lambda)}(x,y) \,  \dd y + \lambda \int^{ap^{-1} \vee b}_{b} r^{(q+\lambda)}(x,y) \,  g_{a,b}^{(q)}(py) \dd y, \notag
		\end{align}
	and thus using Theorem \ref{Thm:Res-UAuxiliaryFuncIsFixedPoint}, we get that
	 			\begin{align}
		g_{a,b}(x)	&= \bigl[ {W}^{(q+\lambda)}(a-b) \bigr]^{-1}\Bigl[{W}^{(q+\lambda)}(x-b)+ \int _b^{a\vee bp^{-1}} \overline{\V}^{(p)}_\lambda[r^{(q+\lambda)}](x,y) {W}^{(q+\lambda)}(py-b) \dd y  \Bigr]+\mathcal H_{p}^{(q+\lambda)}(x;b,a), \notag 
		\end{align}
		where 
			\begin{equation*}
			\mathcal H_{p}^{(q+\lambda)}(x;b,a)=\lambda\Bigl[\int_{ap^{-1}\vee b}^{a}r^{(q+\lambda)}(x,y)\dd y+ \int_b^{ap^{-1}\vee b}\int_{ap^{-1}\vee b}^{a}\overline{\V}^{(p)}_\lambda[r^{(q+\lambda)}](x,y) 	r^{(q+\lambda)}(py,z)\dd z \dd  y\Bigr]. 
			\label{Def:H} 
		\end{equation*} 
	In the same fashion with (i),    the result follows immediately after noticing that using Definition \ref{Def:ResolventDefn} it holds that 
	\begin{align*}
r^{(q+\lambda)}(x,y)+ \int_b^{ap^{-1}\vee b}\overline{\V}^{(p)}_\lambda[r^{(q+\lambda)}](x,y) 	r^{(q+\lambda)}(py,z)\dd  y&={\V}^{(p)}_1[r^{(q+\lambda)}](x,y) +\sum_{k=1}^\infty \lambda ^k{\V}^{(p)}_{k+1}[r^{(q+\lambda)}](x,y) \\
&=\frac{1}{\lambda}\overline{\V}^{(p)}_\lambda[r^{(q+\lambda)}](x,y), 
		\end{align*}
		and therefore $\mathcal H_{p}^{(q+\lambda)}$ reduces to the form of the theorem. 
	\end{proof}
\noindent 	We continue by studying downward two-sided exit problems.
\begin{theorem}[Downward two-sided exit problem]\label{Thm:Res-TwoSidedExitDownwards-PositiveHalfAxis}
		Let $p \in (0,1)$,  $q >0 $ and $\lambda \geq 0$. Then,  $x \in [b,a]$ and $a>b \in \R$,    the following identity holds 
	\begin{align}
	\E_x \Bigl( \ee^{-q \tau_{b,U}^-} \1_{\{\tau_{b,U}^- < \tau_{a,U}^+\}} \Bigr) =\mathcal {Z}_{b,a}^\p r^{(q+\lambda)}(x)-\frac{\mathcal{W}_{b,a}^\p r^{(q+\lambda)}(x)}{\mathcal{W}_{b,a}^\p r^{(q+\lambda)}(a)}\mathcal {Z}_{b,a}^\p r^{(q+\lambda)}(a)+\int_{b}^{a\wedge bp^{-1}}\overline{\V}^{(p)}_\lambda[r^{(q+\lambda)}](x,y) 	\dd y\1_{\{a, b\in \R_+\}},
	\label{Eq:PSR-ExitDownwards-abPos}
\end{align}
	where  
	the operator 	$r^{(q)}$   given in   \textbf{\upshape {(S1)}}.
	\end{theorem}	
	\begin{proof}
			Let $a,b \in \R$ and  $h_{a,b}^{(q)}(x) := 	\E_x \bigl( \ee^{-q \tau_{b,U}^-} \1_{\{\tau_{b,U}^- < \tau_{a,U}^+\}} \bigr) $. Similar  line of logic as in  Eq. \eqref{Eq:PSR-ExitUpwardsMainRenewal} of Theorem   \ref{Thm:Res-TwoSidedExitUpwards-PositiveHalfAxis} [using  Eq. \eqref{eq:ClassicalExitfromBelow} for  the downward case], we have 
		\begin{align}
			h_{a,b}^{(q)}(x) 
			&=  Z^{(q+\lambda)}(x-b) - \frac{W^{(q+\lambda)}(x-b)}{W^{(q+\lambda)}(a-b)}  Z^{(q+\lambda)}(a-b) + \lambda \int^a_{b} r^{(q+\lambda)}(x,y) \, h_{a,b}^{(q)}(py) \dd y.  \label{Eq:PSR-ExitDownwardsMainRenewal}
		\end{align}
		\noindent 
	(i) Consider $a,b \in \R_+$. Noticing,  for $a > bp^{-1}$ that  $h_{a,b}^{(q)}(py) = 1$ for $y<bp^{-1}$ and for $bp^{-1} \geq a >b$  that $h_{a,b}^{(q)}(py) = 1$ for all $y \in [b,bp^{-1})$, Eq. \eqref{Eq:PSR-ExitDownwardsMainRenewal} becomes 
			\begin{align}
			h_{a,b}^{(q)}(py) =&	 Z^{(q+\lambda)}(x-b) - \frac{W^{(q+\lambda)}(x-b)}{W^{(q+\lambda)}(a-b)}  Z^{(q+\lambda)}(a-b) +  \lambda \int^{a\wedge bp^{-1}}_b r^{(q+\lambda)}(x,y) \, \dd y \notag \\
& + \lambda \int^a_{a\wedge bp^{-1}} r^{(q+\lambda)}(x,y) \, h_{a,b}^{(q)}(py) \dd y. 		\notag 
				 		\end{align}	
			 		Hence, from Theorem \ref{Thm:Res-UAuxiliaryFuncIsFixedPoint}, we have that 
		 				\begin{align}
		 				h_{a,b}^{(q)}(py) 
		 				=&Z^{(q+\lambda)}(x-b)+ \int^a_{a\wedge bp^{-1}}	\overline{\V}^{(p)}_\gamma[r^{(q)}](x,y)Z^{(q+\lambda)}(py-b)\dd y\notag \\
		 				&-\bigl[W^{(q+\lambda)}(a-b)\bigr]^{-1}\Bigl[W^{(q+\lambda)}(x-b)+\int^a_{a\wedge bp^{-1}}	\overline{\V}^{(p)}_\gamma[r^{(q)}](x,y)W^{(q+\lambda)}(py-b)\dd y\Bigr]Z^{(q+\lambda)}(a-b)
\notag \\		 				
		 				& +\mathcal{K}_p^{(q+\lambda)}(x;b,a),\notag 	\end{align}	
		 				with
		 					\begin{equation*}
		 					\mathcal{K}_p^{(q+\lambda)}(x;b,a)=\lambda\Bigl[\int_{b}^{a\wedge bp^{-1}}r^{(q+\lambda)}(x,y)\dd y+ \int_{a \wedge bp^{-1}}^{a}\int_{b}^{a \wedge bp^{-1}}\overline{\V}^{(p)}_\lambda[r^{(q+\lambda)}](x,y) 	r^{(q)}(py,z)\dd z \dd  y\Bigr]. 
		 				\end{equation*}
	 			Treating $\mathcal{K}_p^{(q+\lambda)}$	in a similar manner with the proof of Theorem \ref{Thm:Res-TwoSidedExitUpwards-PositiveHalfAxis}(iii) and using the definition of the (partial resetting) scale functions of Eq. \eqref{Eq:PSR-ScaleFunc} for $a$, $b\in\mathbb \R_+$,  the results follows. 	
	 				\vspace{0.2cm}\\
	 				\noindent 	(ii) Let $a\in \mathbb R_+$, $b\in \mathbb R_-$. This case can  be shown using  similar arguments to (i) with the only difference that the integral limits are from $a$ to $b$.
	 				\vspace{0.2cm}\\
	 				\noindent 	(iii) Consider,  $a$, $b\in \mathbb R_+$. Noticing  for  $ap^{-1} > b$ that  $h_{a,b}^{(q)}(py) = 0$ for $y \in [ap^{-1},a]$, and for   $ap^{-1} \leq b < a$  that $h_{a,b}^{(q)}(py) = 0$ for all $y \in [b,a)$, Eq. \eqref{Eq:PSR-ExitDownwardsMainRenewal} translates to 
	 					\begin{align}
	 					h^{(q)}_{a,b}(x) &= Z^{(q+\lambda)}(x-b) - \frac{W^{(q+\lambda)}(x-b)}{W^{(q+\lambda)}(a-b)}Z^{(q+\lambda)}(a-b) + \lambda \int^{ap^{-1} \vee b}_{b} r^{(q+\lambda)}(x,y) \, h_{a,b}^{(q)}(py) \dd y.  \notag 
	 				\end{align}
 				Then using the same technique as in (i), the result follows. 
	\end{proof}
 We proceed with the  the one-sided exit problems. We note that 	
 the classical means of determining the one-sided limits of SNLPs involves using the two-sided exit identities, the spatial invariance property of L\'evy processes and the well-known limit results from Eq.~\eqref{eq:LimitofRatioofScaleFuncs}. However, in our case, the PSR-LP lacks spatial invariance since the resetting occurs w.r.t.~the position of the $x$-axis which is fixed. Therefore, we shall use  Theorem \ref{Thm:Res-UAuxiliaryFuncIsFixedPoint} to ensure appropriate starting kernels can be used that correspond to the necessary one-sided exit problems. 	For these reasons, the proofs of the one-sided exit problems will be written briefly to avoid the necessary repetitions. 
 \begin{proposition}[Downward one-sided exit problem]\label{Thm:Res-OneSidedExitDownwards-PositiveHalfAxis}
 	Let $b\in \mathbb R$, $p \in (0,1)$, $q >0 $ and $\lambda \geq 0$. Then, it holds that 
 		\begin{equation*}
 		\E_x \Bigl( \ee^{-q \tau_{b,U}^-} \1_{\{\tau_{b,U}^- < \infty\}} \Bigr) 
 		=\mathcal{Z}_{b,\infty}^{p}\underline r^{(q+\lambda)}(x) -\frac{q+\lambda}{\Phi_{q+\lambda}}  \mathcal{W}_{b,\infty}^{p}\underline r^{(q+\lambda)}(x)+\int^{b}_{ bp^{-1}} \overline{\V}^{(p)}_\lambda[\underline{r}^{(q+\lambda)}](x,y)  \dd y\mathbf 1_{\{b\in\mathbb R_+\}}, \label{Eq:PSR-VScaleFunc}  
 	\end{equation*}
 	with $\underline {r}^{(q+\lambda)}$  given in \textbf{\upshape{(S1)}}. 
 \end{proposition}
 \begin{proof}Let $b \in \R$ and  $h_{b}^{(q)}(x) := 	\E_x ( \ee^{-q \tau_{b,U}^-} \1_{\{\tau_{b,U}^- < \infty\}} ) $, where $x \in [b,\infty)$. Then, by conditioning on the next Poisson epoch and using the strong Markov property, Eqs \eqref{eq:LimitofRatioofScaleFuncs}  and \eqref{eq:ClassicalExitfromBelow}, we have that 
 	\begin{align}
 		h_b^{(q)}(x)
 		&= \E_{x} \Bigl( \ee^{-(q+\lambda) \tau_{b,X}^-} \1_{\{\tau_{b,X}^- < \infty\}} \Bigr) + \lambda \int^\infty_{b} \E_x \Bigl( \int^\infty_0 \ee^{-(q+\lambda)t} \1_{\{X_t \in \dd y, \, t < \tau_{b,X}^- \}} \dd t \Bigr) h_b^{(q)}(py) \notag \\
 		&= Z^{(q+\lambda)}(x-b) - \frac{q+\lambda}{\Phi_{q+\lambda}} W^{(q+\lambda)}(x-b)+ \lambda \int_b^{\infty}  \underline{r}^{(q+\lambda)}(x,y)   h_b^{(q)}(py) \dd y.\notag 
 	\end{align}
 Noticing that for $b\in\mathbb R_+$, $h_b^{(q)}(py) = 1$ for $y \in [b,bp^{-1}]$, the above Eq.  can be  written universally,  for $b\in \mathbb R$, as 
 	\begin{align}
 	h_b^{(q)}(x)
 	&= \mathcal{V}^{(q+\lambda)}_{p}(x,b)+ \lambda \int_{b\vee bp^{-1}}^{\infty}  \underline{r}^{(q+\lambda)}(x,y)  h_b^{(q)}(py) \dd y, \notag 
 \end{align}
 with 
 \begin{equation*}
 	\mathcal{V}^{(q+\lambda)}_{p}(x,b) =  Z^{(q+\lambda)}(x-b) - \frac{q+\lambda}{\Phi_{q+\lambda}} W^{(q+\lambda)}(x-b) + \lambda \int^{b\vee bp^{-1}}_{b} \underline{r}^{(q+\lambda)}(x,y) \dd y. \label{Eq:TR-VScaleFunc} 
 \end{equation*}
Furthermore, observe that $\mathcal{V}_{p}^{(q+\lambda)}(x,b) \in  C_\mathrm{b}([b,\infty))$ which follows by Eq.~\eqref{Eq:ResolventInequality}.   Using Theorem \ref{Thm:Res-UAuxiliaryFuncIsFixedPoint}, we have 
\begin{align*}
		h_b^{(q)}(x)
&=Z^{(q+\lambda)}(x-b)+   \int^{\infty}_{b\vee bp^{-1}} \overline{\V}^{(p)}_\lambda[\underline{r}^{(q+\lambda)}](x,y)Z^{(q+\lambda)}(py-b) \dd y   \\
&\quad - \frac{q+\lambda}{\Phi_{q+\lambda}} W^{(q+\lambda)}(x-b)- \frac{q+\lambda}{\Phi_{q+\lambda}} \int^{\infty}_{b\vee bp^{-1}} \overline{\V}^{(p)}_\lambda[\underline{r}^{(q+\lambda)}](x,y)  W^{(q+\lambda)}(py-b)\dd y \\
&\quad +\int^{ bp^{-1}}_{b}  \overline{\V}^{(p)}_\lambda[\underline{r}^{(q+\lambda)}](x,y)\dd y\mathbf 1 _{\{b\in \mathbb R_+\}},	\end{align*}
where the last term follows in the same manner as in the proof of Theorem \ref{Thm:Res-TwoSidedExitUpwards-PositiveHalfAxis}(iii). 
 \end{proof}
 
 	\begin{proposition}[Upward one-sided exit problem]\label{Thm:Res-OneSidedExitUpwards-PositiveHalfAxis}
 	Let $a\in \mathbb  R$, $p \in (0,1)$, $q >0 $ and $\lambda \geq 0$. Then, it holds that 
 		\begin{equation*}
 		\E_x \Bigl( \ee^{-q \tau_{a,U}^+} \1_{\{\tau_{a,U}^+ < \infty\}} \Bigr) 
 		= \ee^{-\Phi_{q+\lambda}(a-x)} + \int_{-\infty}^{a\wedge ap^{-1}} \overline{\V}^{(p)}_\lambda[\overline{r}^{(q+\lambda)}](x,y) \ee^{-\Phi_{q+\lambda}(a-py)} \dd y+\int^{a}_{ ap^{-1}} \overline{\V}^{(p)}_\lambda[\overline{r}^{(q+\lambda)}](x,y)  \dd y\mathbf 1_{\{a\in\mathbb R_-\}},  \label{Eq:PSR-ExitUpwardsOneSided-abPos} 
 	\end{equation*}
 	where 
 $\overline {r}^{(q+\lambda)}$  given in \textbf{\upshape{(S1)}}. 
 \end{proposition}
 \begin{proof}
	Let $a \in \R$  and denote  $g_{a}^{(q)}(x) := 	\E_x \bigl( \ee^{-q \tau_{a,U}^+} \1_{\{\tau_{a,U}^+ < \infty\}} \bigr) $.  Then, using similar arguments as  above, we have that 
	  	\begin{align}
	  		g_a^{(q)}(x)	&= \E_{x} \bigl( \ee^{-(q+\lambda) \tau_{a,X}^+} \1_{\{\tau_{a,X}^+ < \infty\}} \bigr) + \lambda \int^a_{-\infty} \E_x \Bigl( \int^\infty_0 \ee^{-(q+\lambda)t} \1_{\{X_t \in \dd y, \, t < \tau_{a,X}^+ \}} \dd t \Bigr) g_a^{(q)}(py)\dd y  \notag \\
	  		&= \ee^{-\Phi_{q+\lambda}(a-x)}+\lambda \int_ {a\wedge ap^{-1}}^a\overline r ^{(q+\lambda)}(x,y)\dd y+ \lambda \int_{-\infty} ^{a\wedge ap^{-1}} \overline r ^{(q+\lambda)}(x,y) \, g_a^{(q)}(py) \dd y. \notag  
	  	\end{align}
  	Now, observe that the function of the first two terms belong to $ C_\mathrm{b}((-\infty,a])$, and applying  Theorem  \ref{Thm:Res-UAuxiliaryFuncIsFixedPoint} and using the same  arguments as in  Proposition \ref{Thm:Res-OneSidedExitDownwards-PositiveHalfAxis}, gives the required result. 
 \end{proof}
	\subsection{Exit identities under reflection}\label{reflection}
In this subsection we extend the fluctuation identities derived previously to partially resetting Lévy processes subject to Skorokhod reflection;  both lower and upper.  The proofs follow the same integral-equation approach as in the previous subsection, with the potential densities of the reflected Lévy processes serving as the starting kernels in the resolvent construction. 
	
	\noindent 	Let  $X^b\equiv\{X^b_t\}_{t \geq 0}$, with  $  X^b_t=X_t+\sup _{0\leq s\leq t}(b-X_s)\vee 0$, be a SNLP reflected at  lower level $b$. The supremum term pushes the process upward whenever it attempts to down-cross the level $b$. It is known, see Theorem 1 in \cite{P2004}, that  the exit upwards and the  potential density killed on exiting $a$ are  given  by 
		\[
\mathbb 	E_x(\ee ^{-q\tau^+_{a,X ^b}})=\frac{Z^{(q+\lambda)}(x-b)}{Z^{(q+\lambda)}(a-b)}	, \quad r_b^{(q)}(x,y)=\frac{Z^{(q)}(x-b)}{Z^{(q)}(a-b)}W^{(q)}(a-y)-W^{(q)}(x-y),\quad x,y\in[b,a).
		\]
		 Furthermore, we define the reflected at a lower level $b$ partial resetting process, $U^b\equiv\{U^b_t\}_{t \geq 0}$,  given by 
		\[
			\dd  U^b_t=\dd X_t-(1-p)U^b_{t^{-}}\dd N_t+\dd R^b_t
		\]
		where the downward Skorokhod regulator is  $R^b_t=\sup_{0\leq s\leq t }\{b-X_s+(1-p)\int_{0}^{s}U^b_{u^-}\dd N_u\}\vee 0 $ is a right continuous, nondecreasing process that pushes  upwards $U^b$ whenever it attempts to cross $b$. 
	\begin{remark}\label{remref}\upshape{
		Let $0=T_0<T_1<T_2<\cdots$ denote the Poisson arrival times. Between two consecutive Poisson epochs, no resetting occurs and the process evolves as the Lévy process reflected at the lower barrier $b$. Consequently, for
		$t\in[T_{k-1},T_k)$,
		\[
		U_t^b
		=
		U_{T_{k-1}}^b
		+
		X_t-X_{T_{k-1}}
		+
		\sup_{T_{k-1}\le s\le t}
	\{	
		b-U_{T_{k-1}}^b-(X_s-X_{T_{k-1}})
		\}\vee 0.
		\]
		Thus, on each interval $[T_{k-1},T_k)$ the path coincides with the Skorokhod reflection of the Lévy process started from $U_{T_{k-1}}^b$.
		
	\noindent 	At the Poisson epoch $T_k$, the process is partially reset according to
	$
		U_{T_k}^b
		=
		b\vee(pU_{T_k-}^b),
	$
		that is  whenever the reset would place it below $b$, it is instantaneously pushed back to the barrier.
			Hence, the process is obtained recursively by alternating reflected Lévy evolution between Poisson epochs with the reset rule above. Since the Skorokhod reflection map is pathwise unique and the reset mapping $x\mapsto b\vee(px)$ is deterministic and Lipschitz, the SDE admits a unique strong solution by the same recursive argument as in Theorem    \ref{sdepart}. }

			\end{remark}
				\begin{proposition}[First passage time for the reflected process at  lower level]\label{upreflection}
				Let $a$, $b\in \mathbb  R$, $p \in (0,1)$, $q >0 $ and $\lambda \geq 0$. Then, for  $x\in[b,a)$, it holds that 
				\begin{align*}
					\mathbb 	E_x(\ee ^{-q\tau^+_{a,U^b }})&=\frac{{\mathcal{Z}_{b,a}^\p}r_b^{(q+\lambda )}(x)}{	{\mathcal{Z}_{b,a}^\p}r_b^{(q+\lambda )}(a)}+\int_{ap^{-1}\vee b}^{a}\overline{\V}^{(p)}_\lambda[r_b^{(q+\lambda)}](x,y) \dd y\mathbf 1 _{\{a,b\in \mathbb R_-\}}\\
					&\quad + \ell_a^{(p)}(b)\int_b^{a\wedge bp^{-1}}\overline{\V}^{(p)}_\lambda[r_b^{(q+\lambda)}](x,y)\dd y\mathbf 1 _{\{a,b\in \mathbb R_+\}}, 
				\end{align*}
				where  
				\begin{equation*}
					\ell_a^{(q)}(b)
					=
				\Bigl[Z^{(q+\lambda)}(a-b)
				\Bigl(
				1-
				\displaystyle\int_b^{a\wedge bp^{-1}}
				\overline{\V}^{(p)}_\lambda
				[r_b^{(q+\lambda)}](b,z)\,\mathrm{d}z
				\Bigr)\Bigr]^{-1}	\Bigl[
						1+
						\displaystyle\int_{a\wedge bp^{-1}}^a
						\overline{\V}^{(p)}_\lambda
						[r_b^{(q+\lambda)}](b,y)
						Z^{(q+\lambda)}(py-b)\,\mathrm{d}y
\Bigr].				\end{equation*}
			\end{proposition}
\begin{proof}	\noindent 	 Let  
		 $a$, $b\in \mathbb R $ and $\ell_a^{(q)}(x)=\mathbb E_x(\ee ^{-q\tau^+_{a,U ^b}})$. Similarly as before, 
	\begin{align}\label{reflection1}
		\ell_a^{(q)}(x)
		&=\frac{Z^{(q+\lambda)}(x-b)}{Z^{(q+\lambda)}(a-b)}+\lambda \int_b^a r_b^{(q+\lambda )}(x,y)\ell_a^{(p)}(py)\dd y. 
	\end{align}
	(i) Consider $a$, $b \in \R_+$. Then similar line of logic as before can be used and the above equation reduces to 
	\begin{align*}
		\ell_a^{(p)}(x)
		&=\frac{Z^{(q+\lambda)}(x-b)}{Z^{(q+\lambda)}(a-b)}+\lambda \ell_a^{(p)}(b)\int_b^{a\wedge bp^{-1}} r_b^{(q+\lambda)}(x,y)\dd y +\lambda \int_{a\wedge bp^{-1}}^a r_b^{(q+\lambda )}(x,y)\ell_a^{(p)}(py)\dd y. 
	\end{align*}
Using Theorem \ref{Thm:Res-UAuxiliaryFuncIsFixedPoint}, we get that 
	\begin{align*}
	\ell_a^{(q)}(x)
	&=\frac{Z^{(q+\lambda)}(x-b)}{Z^{(q+\lambda)}(a-b)}+\lambda \ell_a^{(p)}(b)\int_b^{a\wedge bp^{-1}} r_b^{(q+\lambda)}(x,y)\dd y \\
	&\quad + \int_{a\wedge bp^{-1}}^a \overline{\V}^{(p)}_\lambda[r_b^{(q+\lambda)}](x,y)\Bigl(\frac{Z^{(q+\lambda)}(py-b)}{Z^{(q+\lambda)}(a-b)}+\lambda \ell_a^{(p)}(b)\int_b^{a\wedge bp^{-1}}r^{(q+\lambda)}_b(py,z)\dd z\bigr)\dd y,
\end{align*}
from which the result follows after some algebraic manipulations. 

\noindent 	(ii) For $a\in \R_+$ and $b \in \R_-$, $	\ell_a^{(q)}(x)$ takes the form of Eq. \eqref{reflection1} and the result follows by similar arguments as in  Theorems \ref{Thm:Res-TwoSidedExitUpwards-PositiveHalfAxis} and  \ref{Thm:Res-TwoSidedExitDownwards-PositiveHalfAxis}. 
\vspace{0.2cm}\\
	\noindent 	(iii) Consider $a$, $b\in \R_-$, then we have that 
	\begin{align*}
	\ell_a^{(q)}(x)
	&=\frac{Z^{(q+\lambda)}(x-b)}{Z^{(q+\lambda)}(a-b)}+\lambda \int_b^{b\vee ap^{-1}} r_b^{(q+\lambda )}(x,y)\ell_a^{(p)}(py)\dd y+\lambda  \int_{b\vee  ap^{-1}}^a r_b^{(q+\lambda )}(x,y)\dd y, 
\end{align*}	
and the result follows in 
a similar manner with the proof of Theorem \ref{Thm:Res-TwoSidedExitUpwards-PositiveHalfAxis}. 
	\end{proof}

		\noindent 	Now, let   $X^a\equiv\{X^a_t\}_{t \geq 0}$ with $  X^a_t=X_t-\sup _{0\leq s\leq t}(X_s-a)\vee 0$ be a SNLP reflected at  upper  level $a$. The supremum term pushes the process downward  whenever it attempts to upward-cross the level $a$.  The corresponding $q$-potential density of the process reflected at the upper level $a$ and killed upon exiting below $b$ is given by, see Theorem 8.11 in \cite{K2014} or \cite{P2004}, for 	$ x,y\in(b,a]$, 
	\[\mathbb{E}_x\left(\mathrm e^{-q\tau_{b,X^a}^-}\right)
	=
	Z^{(q)}(x-b)
	-
	\frac{Z^{(q)\prime}(a-b)}{W^{(q)\prime}(a-b)}W^{(q)}(x-b)
	,\quad 
	r_a^{(q)}(x,y)
	=
	\frac{W^{(q)}(x-b)}{W^{(q)\prime}(a-b)}
	W^{(q)\prime}(a-y)	
	-
	W^{(q)}(x-y), 
	\]
	with $W ^{(q)\prime}$ to mean the right derivative. We note that the above identity for $\mathbb{E}_x(\mathrm e^{-q\tau_{b,X^a}^-})$ is slightly different than the one in Theorem 8.11 in \cite{K2014}, however it follows immediately using Eq.  \eqref{eq:LTofscaleW}. 
		 Furthermore, we define the reflected at a upper level $a$ partial resetting process, $U^a\equiv\{U^a_t\}_{t \geq 0}$,  given by 
	\[
	\dd  U^a_t=\dd X_t-(1-p)U^a_{t^{-}}\dd N_t-\dd R^a_t, 
	\]
	where the  upward  regulator is  $R^a_t=\sup_{0 \leq s\leq t}\{X_s-(1-p)\int_{0}^sU^a_{u^-}\dd N_u-a\}\vee 0$ is a right continuous, nondecreasing process that pushes  upwards $U^a$ whenever it attempts to cross $a$;   we note that the   above  SDE similar remarks holds as in  Remark \ref{remref}. 
		\begin{proposition}[First passage time for the reflected process at upper  level]\label{downreflection}
		Let $a$, $b\in \mathbb  R$, $p \in (0,1)$, $q >0 $ and $\lambda \geq 0$. Then, it holds that 
		\begin{align*}
			\mathbb E_x(\ee ^{-q\tau^-_{b,U^a}})&={\mathcal{Z}_{b,a}^\p}r_a^{(q+\lambda )}(x)-
			\frac{{\mathcal{Z}_{b,a}^{\p\prime} }r_a^{(q+\lambda )}(a)}{{\mathcal{W}_{b,a}^{\p\prime}}r_a^{(q+\lambda )}(a)}{\mathcal{W}_{b,a}^\p}r_a^{(q+\lambda )}(x)+  \int_b^{a\wedge bp^{-1}} \overline{\V}^{(p)}_\lambda[r_a^{(q+\lambda)}](x,y)\dd y\mathbf 1_{\{a,b\in \mathbb R_+\}}\\
			&\quad + \ell_b^{(q)}(a)\int_{ap^{-1}\vee b}^{a}\overline{\V}^{(p)}_\lambda[r_a^{(q+\lambda)}](x,y)\dd y\mathbf 1_{\{a,b\in \mathbb R_-\}},
			\end{align*}
			where ${\mathcal{W}_{b,a}^{\p\prime} }r_a^{(q+\lambda )}(a)=\frac{\dd }{\dd x }{\mathcal{W}_{b,a}^{\p} }r_a^{(q+\lambda )}(x)|_{x=a}$, similarly for  ${\mathcal{Z}_{b,a}^{\p\prime} }r_a^{(q+\lambda )}(a)$  and 
			\begin{align*}
				\ell_b^{(q)}(a)
				=
			\Bigl[1-
			\displaystyle\int_{ap^{-1}\vee b}^{a}
			\overline{\mathcal V}^{(p)}_\lambda
			[r_a^{(q+\lambda)}](a,y)\,\mathrm{d}y\Bigr]^{-1}\Bigl[	{
					{\mathcal{Z}_{b,a}^{p}}r_a^{(q+\lambda)}(a)
					-
					\dfrac{
						{\mathcal{Z}_{b,a}^{p\prime}}r_a^{(q+\lambda)}(a)
					}{
						{\mathcal{W}_{b,a}^{p\prime}}r_a^{(q+\lambda)}(a)
					}
					{\mathcal{W}_{b,a}^{p}}r_a^{(q+\lambda)}(a)
				}\Bigr].
			\end{align*}
	\end{proposition}
		\begin{proof}
			For   
			$a$, $b\in \mathbb R $, let  $\ell_b^{(q)}(x):=\mathbb E_x(\ee ^{-q\tau^-_{b,U^a}})$. Similarly as before, 
			\begin{align*}
				\ell_b^{(q)}(x)
				&=Z^{(q+\lambda)}(x-b)
					-
				\frac{Z^{(q+\lambda)\prime }(a-b)}{W^{(q+\lambda)\prime}(a-b)}W^{(q+\lambda)}(x-b)+\lambda \int_b^a r_a^{(q+\lambda )}(x,y)\ell_b^{(p)}(py)\dd y. 
			\end{align*}
			(i) Consider $a$, $b \in \R_+$ and therefore the above equation in reduced to 
			\begin{align*}
				\ell_b^{(q)}(x)
				&=Z^{(q+\lambda)}(x-b)
				-
				\frac{Z^{(q+\lambda)\prime }(a-b)}{W^{(q+\lambda)\prime}(a-b)}W^{(q+\lambda)}(x-b)+\lambda \int_b^{a\wedge bp^{-1}} r_a^{(q+\lambda )}(x,y)\dd y\\
				&\quad + \lambda \int_{a\wedge bp^{-1}}^a r_a^{(q+\lambda )}(x,y)\ell_b^{(p)}(py)\dd y, 
			\end{align*}
			which, similar  as before  has solution of the form 
			\begin{align*}
				\ell_b^{(q)}(x)
				&={\mathcal{Z}_{b,a}^\p}r_a^{(q+\lambda )}(x)	-
				\frac{Z^{(q+\lambda)\prime }(a-b)}{W^{(q+\lambda)\prime}(a-b)}{\mathcal{W}_{b,a}^\p}r_a^{(q+\lambda )}(x)+  \int_b^{a\wedge bp^{-1}} \overline{\V}^{(p)}_\lambda[r_a^{(q+\lambda)}](x,y)\dd y. 
			\end{align*}
			Noticing that $
			\frac{\partial}{\partial x}
			r_a^{(q+\lambda)}(x,y)
			\bigl|_{x=a}
			=
			W^{(q+\lambda)\prime}(a-y)
			-
			W^{(q+\lambda)\prime}(a-y)
			=
			0, 
			$ gives 	$
			\frac{\partial}{\partial x}
			\overline{{\V}}_{\lambda}^{(p)}
			[r_a^{(q+\lambda)}](x,y)
			\bigl|_{x=a}
			=0,
			$
from which it follows that 			
	\[
\mathcal{W}_{b,a}^{p \prime}[r_a^{(q+\lambda)}]
(a)
=
W^{(q+\lambda)\prime}(a-b) \quad \text{and} \quad  \mathcal{Z}_{b,a}^{p \prime}[r_a^{(q+\lambda)}]
(a)
=
Z^{(q+\lambda)\prime}(a-b). 
\]
Combining the last two eqs the result follows immediately. 

\noindent 				(ii) The case  $a\in \mathbb R_-$, $b \in \R_+$
is exactly similar to (i). 
			
\noindent 				(iii) For   $a$, $b\in \mathbb R_-$,  we have that 
	\begin{align*}
	\ell_b^{(q)}(x)
	&=Z^{(q+\lambda)}(x-b)
	-
	\frac{Z^{(q+\lambda)\prime}(a-b)}{W^{(q+\lambda)\prime}(a-b)}W^{(q+\lambda)}(x-b)+\lambda \int_b^{ap^{-1}\vee  b} r_{a}^{(q+\lambda )}(x,y)\ell_b^{(q)}(py)\dd y\\
	&\quad +\lambda \ell_b^{(q)}(a)\int_{ap^{-1}\vee  b}^{a}r_{a}^{(q+\lambda )}(x,y)\dd y, 
\end{align*}
			from which, similar as before,  the result  follows.
			\end{proof}

\subsection{The total resetting as a limiting case }\label{limitcase}
In this section, we prove the limiting case $p \rightarrow 0$ which coincides with the so-called total stochastic resetting case. To do this, we  first  show that the resolvent series $\overline{\V}_\lambda^{(p)}[\f^{(q+\lambda)}]$ converges  as $p \rightarrow 0$ for all choices of sets $\B$ with  assumptions \textbf{(H1)}  and resolvent densities $\f^{(q+\lambda)}$ introduced in Section  \ref{prelim}. This is the purpose of the next theorem.
\begin{lemma}\label{Thm:ResolventLimitPto0Exists}
	Let $q,\lambda >0$, $p \in (0,1)$ and $[u_1,u_2] \subset \B$ s.t.~$[pu_1,pu_2] \subset \B$. Then, for $x \in \B$ and $f \in C_\mathrm{b}(\B)$, 
	\begin{equation}
		\lim\limits_{p\rightarrow 0} \int^{u_2}_{u_1} \overline{\V}^{(p)}_\lambda[\f^{(q+\lambda)}](x,y) f(py) \dd y  = f(0)  \frac{\lambda \int^{u_2}_{u_1} \f^{(q+\lambda)}(x,y) \dd y}{1- \lambda \int^{u_2}_{u_1} \f^{(q+\lambda)}(0,y) \dd y}, 
	\end{equation} 
	where $\overline{\V}^{(p)}_\lambda[\f^{(q+\lambda)}]$ is given in Definition \ref{Def:ResolventDefn}.
\end{lemma}


\begin{proof} We shall  show by induction that,  
for $k \geq 1$,    $\lim_{p \rightarrow 0} \lambda^k \int^{u_2}_{u_1} \V^{(p)}_k[\f^{(q+\lambda)}] (x,y) f(py)\, \dd y$ exists. For $k = 1$ we have  
	\begin{equation}
		\lim\limits_{p \rightarrow 0} \lambda \int^{u_2}_{u_1} \V^{(p)}_1[\f^{(q+\lambda)}] (x,y) f(py) \dd y =  f(0) \cdot \lambda \int^{u_2}_{u_1} \f^{(q+\lambda)} (x,y)\, \dd y, \label{Eq:CaseFork=1-pLimit}
	\end{equation}
	which follows by  Remark \ref{rem2} and dominated convergence  theorem.  Now, for arbitrary $k \geq 2$,  using Definition \ref{Def:ResolventDefn},  Fubini's Theorem, Remark \ref{rem2} and dominated convergence, we have that 
	\begin{align}
		\lim\limits_{p\rightarrow 0}\lambda^{k} \int^{u_2}_{u_1} \V^{(p)}_k[\f^{(q+\lambda)}] (x,y) f(py) \, \dd y &= 	f(0) \cdot \lambda \int^{u_2}_{u_1} \f^{(q+\lambda)}(x,z) \,\Bigl( \lim\limits_{p\rightarrow 0} \, \lambda^{k-1} \int^{u_2}_{u_1}  \V^{(p)}_{k-1}[\f^{(q+\lambda)}] (pz,y)  \, \dd y \Bigr) \, \dd z \notag \\
		&= 	f(0) \cdot \lambda \int^{u_2}_{u_1} \f^{(q+\lambda)}(x,z) \, \Bigl( \lim\limits_{p\rightarrow 0} \, \lambda^{k-1} \int^{u_2}_{u_1}  \V^{(p)}_{k-1}[\f^{(q+\lambda)}] (0,y)  \, \dd y \Bigr) \, \dd z, \label{Eq:Recursive-p-Limit}
	\end{align}
	where the last equality holds by recalling that $\lambda^{k-1} \int^{u_2}_{u_1}  \V^{(p)}_{k-1}[\f^{(q+\lambda)}] (x,y)  \, \dd y \in C_\mathrm{b}(\B)$ for $k \geq 2$. Now, by using Eqs.~\eqref{Eq:CaseFork=1-pLimit} and \eqref{Eq:Recursive-p-Limit} above with $f = 1$ and $x=0$, we obtain inductively for $k \geq 1$ that 
	\[
	\lim\limits_{p\rightarrow 0}\lambda^{k} \int^{u_2}_{u_1} \V^{(p)}_k[\f^{(q+\lambda)}] (0,y) \, \dd y =  \Bigl( \lambda \int^{u_2}_{u_1} \f^{(q+\lambda)}(0,z) \, \dd z \Bigr)^k,
	\]
	and thus  Eq.~\eqref{Eq:Recursive-p-Limit} becomes 
	\begin{align}
		\lim\limits_{p\rightarrow 0}\lambda^{k} \int^{u_2}_{u_1} \V^{(p)}_k[\f^{(q+\lambda)}] (x,y) f(py) \, \dd y 
		&= 	f(0) \cdot \Bigl(\lambda \int^{u_2}_{u_1} \f^{(q+\lambda)}(x,z) \, \dd z \Bigr) \cdot \Bigl( \lambda \int^{u_2}_{u_1} \f^{(q+\lambda)}(0,z) \, \dd z \Bigr)^{k-1}, \notag
	\end{align}
	for all $k \geq 1$. 
	Therefore, from the above equation, we have for all $k \geq 1$ that
	\begin{align}
		\sum^\infty_{k=1} \lim\limits_{p\rightarrow 0}\lambda^{k} \int^{u_2}_{u_1} \V^{(p)}_k[\f^{(q+\lambda)}] (x,y) f(py) \, \dd y &= f(0) \cdot \Bigl(\lambda \int^{u_2}_{u_1} \f^{(q+\lambda)}(x,z) \, \dd z \Bigr) \sum^{\infty}_{k=1} \Bigl( \lambda \int^{u_2}_{u_1} r^{(q+\lambda)}(0,z) \, \dd z \Bigr)^{k-1} \notag \\
		&= f(0) \cdot \frac{\lambda \int^{u_2}_{u_1} \f^{(q+\lambda)}(x,z) \, \dd z}{1-\lambda \int^{u_2}_{u_1} \f^{(q+\lambda)}(0,z) \, \dd z}, \notag
	\end{align}
	where by the convergence follows from the geometric series since $\lambda \int^{u_2}_{u_1} \f^{(q+\lambda)}(0,z) \, \dd z \leq \frac{\lambda}{q+\lambda}$. Lastly, by utilising the above equation and uniform convergence of the series to interchange the limit and summation, we have  that
	\begin{align}
		\lim\limits_{p\rightarrow 0}\int^{u_2}_{u_1} \overline{\V}_\lambda^{(p)}[\f^{(q+\lambda)}](x,y) \dd y &=  \sum^\infty_{k=1} \lim\limits_{p\rightarrow 0} \, \lambda^k \int^{u_2}_{u_1} \V^{(p)}_k[\f^{(q+\lambda)}] (x,y) \dd y \notag \\
		&= f(0) \cdot \frac{\lambda \int^{u_2}_{u_1} \f^{(q+\lambda)}(x,z) \, \dd z}{1-\lambda \int^{u_2}_{u_1} \f^{(q+\lambda)}(0,z) \, \dd z}, \notag
	\end{align}
	which completes the proof.
\end{proof}

\noindent We note that similar to Remark \ref{valuesofr}, we shall use the  the forms in \textbf{(S1)} for $\f^{(q)}(x,y)$  in Lemma \ref{Thm:ResolventLimitPto0Exists}, depending on the exit problem under consideration. 

 We now define the corresponding scale functions for the  total resetting environment. Using Lemma  \ref{Thm:ResolventLimitPto0Exists}, and taking $p \rightarrow 0$ to the partial resetting scale function operators in Eq.~\eqref{Eq:PSR-ScaleFunc}, 
we have,  for $a, b \in \mathbb{R}$ s.t. $x \in[b, a]$, that 
	\begin{align}
	\begin{split}	\mathcal{W}_{b,a}^{(q+\lambda)}(x) & =W^{(q+\lambda)}(x-b)+W^{(q+\lambda)}(-b) \cdot \frac{\lambda \int_b^a \f^{(q+\lambda)}(x, y) \mathrm{d} y}{1-\lambda \int_b^a \f^{(q+\lambda)}(0, y) \mathrm{d} y} \mathbf{1}_{\left\{a \in \mathbb{R}_{+}, b \in \mathbb{R}_{-}\right\}}, \\
		\mathcal{Z}_{b,a}^{(q+\lambda)}(x) & =Z^{(q+\lambda)}(x-b)+Z^{(q+\lambda)}(-b) \cdot \frac{\lambda \int_b^a \f^{(q+\lambda)}(x, y) \mathrm{d} y}{1-\lambda \int_b^a \f^{(q+\lambda)}(0, y) \mathrm{d} y}\mathbf{1}_{\left\{a \in \mathbb{R}_{+}, b \in \mathbb{R}_{-}\right\}},
		\label{ap3}
	\end{split}
\end{align} 
be the scale  functions of the total resetting L\'evy process.

The next  the result gives expressions for the  two-sided exit upwards and downwards identities of the TSR-LP.

\begin{proposition}[Upward and downward two-sided exit problem]
	\label{ts1}
	For $a, b \in \mathbb{R}$ and $x \in[b, a]$, the following limits hold
	\[
	\begin{aligned}
		\lim _{p \rightarrow 0} \mathbb{E}_x\left(\mathrm{e}^{-q \tau_{a, U}^{+}} \mathbf{1}_{\left\{\tau_{a, U}^{+}<\tau_{b, U}^{-}\right\}}\right)= & \frac{	\mathcal{W}_{b,a}^{(q+\lambda)}(x)}{	\mathcal{W}_{b,a}^{(q+\lambda)}(a)}+\lambda \int_b^a r^{(q+\lambda)}(x, y) \mathrm{d} y \mathbf{1}_{\left\{a, b \in \mathbb{R}_{-}\right\}}, \\
		\lim _{p \rightarrow 0} \mathbb{E}_x\left(\mathrm{e}^{-q \tau_{b, U}^{-}} \mathbf{1}_{\left\{\tau_{b, U}^{-}<\tau_{a, U}^{+}\right\}}\right)= & \mathcal{Z}_{b,a}^{(q+\lambda)}(x )-\frac{\mathcal{W}_{b,a}^{(q+\lambda)}(x )}{\mathcal{W}_{b,a}^{(q+\lambda)}(a )} \mathcal{Z}_{b,a}^{(q+\lambda)}(a)  +\lambda \int_b^a r^{(q+\lambda)}(x, y) \mathrm{d} y \mathbf{1}_{\left\{a, b \in \mathbb{R}_{+}\right\}}.
	\end{aligned}
	\]
	\end{proposition}
\begin{proof}
	First note that for  $p \in (0,1]$, we have that $h_1(p):=a \wedge(a p^{-1} \vee b)$ and $h_2(p):=b \vee(a \wedge b p^{-1})$ are continuous for $p \in(0,1]$. In addition, since $p=0$ is a limit point of $( 0,1 ]$, we have $h_1(p) \rightarrow h_1(0)$ and $h_2(p) \rightarrow h_2(0)$ as $p \rightarrow 0$.
	\noindent Now, we shall use the continuity of these functions to determine the limiting case $p \rightarrow 0$ for the two-sided exit identities given in Theorems \ref{Thm:Res-TwoSidedExitUpwards-PositiveHalfAxis} and \ref{Thm:Res-TwoSidedExitDownwards-PositiveHalfAxis}.
	
	\vspace{0.2cm}
\noindent 	(i)  Let $a, b \in \mathbb{R}_{+}$. Observe that $\lim _{p \rightarrow 0}(a \wedge b p^{-1})=a$, $\lim _{p \rightarrow 0} a \wedge(a p^{-1} \vee b)=a$ and $\lim _{p \rightarrow 0} b \vee(a \wedge b p^{-1})= a$. Thus,  we have for $x \in[b, a]$ that
\begin{align*}
	\lim _{p \rightarrow 0} \mathcal{W}_{b,a}^{\p}r^{(q+\lambda)}(x ) & =W^{(q+\lambda)}(x-b), \\
	\lim _{p \rightarrow 0} \mathcal{Z}_{b,a}^{\p}r^{(q+\lambda)}(x ) & =Z^{(q+\lambda)}(x-b), \\
	\lim _{p \rightarrow 0} \int_{b}^{a\wedge bp^{-1}}\overline{\V}^{(p)}_\lambda[r^{(q+\lambda)}](x,y) 	\dd y & =\lambda \int_b^a r^{(q+\lambda)}(x, y) \mathrm{d} y,
\end{align*}
which, after using Eqs. in \eqref{ap3}, yields the required limiting forms of the exit identities.

	\vspace{0.2cm}
\noindent 	(ii) Let $a, b \in \mathbb{R}_{-}$. Observe that $ \lim _{p \rightarrow 0} a \wedge(a p^{-1} \vee b)=b$, and $\lim _{p \rightarrow 0} b \vee(a \wedge b p^{-1})= b$. Thus, for $x \in[b, a]$, we have 
\[
\begin{aligned}
	\lim _{p \rightarrow 0} \mathcal{W}_{b,a}^{\p}r^{(q+\lambda)}(x ) & =W^{(q+\lambda)}(x-b), \\
	\lim _{p \rightarrow 0} \mathcal{Z}_{b,a}^{\p}r^{(q+\lambda)}(x) & =Z^{(q+\lambda)}(x-b), \\
	\lim _{p \rightarrow 0}  \int_{ap^{-1}\vee b}^{a}\overline{\V}^{(p)}_\lambda[r^{(q+\lambda)}](x,y) \dd y & =\lambda \int_b^a r^{(q+\lambda)}(x, y) \mathrm{d} y,
\end{aligned}
\]
and using Eq. \eqref{ap3} the result follows.

\vspace{0.2cm}
\noindent 	(iii)
 Let $a \in \mathbb{R}_{+}$, $b \in \mathbb{R}_{-}$. In addition, $ \lim _{p \rightarrow 0} a \wedge(a p^{-1} \vee b)=b$, and $\lim _{p \rightarrow 0} b \vee(a \wedge b p^{-1})=b$. Then, by using these limits along with using Lemma \ref{Thm:ResolventLimitPto0Exists}, we have for $x \in[b, a]$ that

$$
\begin{aligned}
	& \lim _{p \rightarrow 0} \mathcal{W}_{b,a}^{\p}r^{(q+\lambda)}(x )=W^{(q+\lambda)}(x-b)+W^{(q+\lambda)}(-b) \cdot \frac{\lambda \int_b^a r^{(q+\lambda)}(x, y) \mathrm{d} y}{1-\lambda \int_b^a r^{(q+\lambda)}(0,  y) \mathrm{d} y} \\
	& \lim _{p \rightarrow 0} \mathcal{Z}_{b,a}^{\p}r^{(q+\lambda)}(x ; b, a)=Z^{(q+\lambda)}(x-b)+Z^{(q+\lambda)}(-b) \cdot \frac{\lambda \int_b^a r^{(q+\lambda)}(x, y) \mathrm{d} y}{1-\lambda \int_b^a r^{(q+\lambda)}(0,  y) \mathrm{d} y}
\end{aligned}
$$
and thus, by using the above limits along with Eqs. in \eqref{ap3} the required exit identities are obtained.
	\end{proof}
\noindent 	Similarly for the total resetting exit identities under reflection are given by the following proposition. 
	\begin{proposition}[Upward and downward exit problems under reflection]
		\label{ts3}
		For $a, b \in \mathbb{R}$, the following limits, in terms of the resolvents ${r}_b^{(q+\lambda)}$ and  ${r}_a^{(q+\lambda)}$ defined in Section \ref{reflection},  hold
		\begin{itemize}
			\item[\upshape{(i)}] for  $x\in[b,a)$, 
			\begin{align*}
			\lim_{p\to 0}	\mathbb 	E_x(\ee ^{-q\tau^+_{a,U^b }})=\frac{{\mathcal{Z}_{b,a}^{(q+\lambda )}}(x)}{	{\mathcal{Z}_{b,a}^{(q+\lambda )}}(a)}+\frac{\int_{b}^{a}r_b^{(q+\lambda)}(x,y)\dd y \mathbf 1_{\{a,b\in \mathbb R_+\}}}{{\mathcal{Z}_{b,a}^{(q+\lambda )}}(a)\bigl(1-\lambda \int_{b}^{a}r_b^{(q+\lambda )}(b,y)\dd y \bigr)}+\lambda \int_{b}^{a}r_b^{(q+\lambda )}(x,y)\mathbf 1_{\{a,b\in \mathbb R_-\}},
			\end{align*} 
							\item[\upshape{(ii)}]  for $ x\in(b,a]$, 
							\begin{align*}
	\lim_{p\to 0}\mathbb E_x(\ee ^{-q\tau^-_{b,U^a}})&=\mathcal{Z}_{b,a}^{(q+\lambda )}(x)-\frac{\mathcal{Z}_{b,a}^{(q+\lambda )\prime}(a)}{\mathcal{W}_{b,a}^{(q+\lambda )\prime}(a)}\mathcal{W}_{b,a}^{(q+\lambda )}(x)+\lambda\int_b^ar_a^{(q+\lambda)}(x,y)\dd y \mathbf 1_{\{a,b\in \mathbb R_+\}}\\
	&\quad +\Bigl(\mathcal{Z}_{b,a}^{(q+\lambda )}(a)-\frac{\mathcal{Z}_{b,a}^{(q+\lambda )\prime}(a)}{\mathcal{W}_{b,a}^{(q+\lambda )\prime}(a)}\mathcal{W}_{b,a}^{(q+\lambda )}(a)\Bigr)\frac{\lambda\int_b^ar_a^{(q+\lambda)}(x,y)\dd y}{1-\lambda\int_b^ar_a^{(q+\lambda)}(a,y)\dd y} \mathbf 1_{\{a,b\in \mathbb R_-\}}
								\end{align*}
		\end{itemize}
		\end{proposition}
	\begin{proof}
	(i)	First note that for $a$, $b\in\mathbb R_+$, we have 
		\[\lim _{p \to 0} \ell_a(b)=\frac{1}{Z^{(q+\lambda)}(a-b)\bigl(1-\lambda\int_{b}^{a}r_b^{(q+\lambda)}(b,y)\dd y\bigr)}.\]
Further for $a$, $b\in\mathbb R$ we have that  $\lim _{p \rightarrow 0} \mathcal{Z}_{b,a}^{\p}r^{(q+\lambda)}(x )  =Z^{(q+\lambda)}(x-b)$ and thus using Lemma \ref{Thm:ResolventLimitPto0Exists} in Proposition \ref{upreflection}, the result follows. Similar treatment holds for part (ii).
		\end{proof}

\noindent Finally, we can also use Lemma \ref{Thm:ResolventLimitPto0Exists} to show that the one-sided PSR-LP identities converge when we let $p \rightarrow 0$. The proof of the result below is similar to Proposition \ref{ts1} and is therefore omitted for brevity. 
\begin{proposition}[Upward and downward one-sided exit problem]
	Let  $q,\lambda > 0$,. Then,  the following limit cases, in terms of the resolvents $\overline{r}^{(q+\lambda)}$ and  $\underline{r}^{(q+\lambda)}$ defined in \textnormal{\textbf{(S1)}}, hold 
	\begin{itemize}
		\item[\upshape{(i)}] For $a,b \in \R_+$, 
		\begin{align}
			\lim\limits_{p \rightarrow 0} \E_x \Bigl( \ee^{-q \tau_{a,U}^+} \1_{\{\tau_{a,U}^+ < \infty\}} \Bigr) &= \ee^{-\Phi_{q+\lambda}(a-x)} + \ee^{-\Phi_{q+\lambda}a} \cdot\frac{\lambda \int^{a}_{-\infty} \overline{r}^{(q+\lambda)}(x,y) \dd y}{1- \lambda \int^{a}_{-\infty} \overline{r}^{(q+\lambda)}(0,y) \dd y}, \quad x \in (-\infty,a], \notag \\
			\lim\limits_{p \rightarrow 0} \E_x \Bigl( \ee^{-q \tau_{b,U}^-} \1_{\{\tau_{b,U}^- < \infty\}} \Bigr) &= {Z}^{(q+\lambda)}(x-b) - \frac{q+\lambda}{\Phi_{q+\lambda}} W^{(q+\lambda)}(x-b) + \lambda \int^\infty_b \underline{r}^{(q+\lambda)}(x,y) \dd y, \quad x \in [b,\infty). \notag
		\end{align}
		
		\item[\upshape{(ii)}] For $a,b \in \R_-$,
		\begin{align}
			\lim\limits_{p \rightarrow 0} \E_x \Bigl( \ee^{-q \tau_{a,U}^+} \1_{\{\tau_{a,U}^+ < \infty\}} \Bigr) &= \ee^{-\Phi_{q+\lambda}(a-x)} + \lambda \int^{a}_{-\infty} \overline{r}^{(q+\lambda)}(x,y) \dd y, \quad x \in (-\infty,a], \notag \\
			\lim\limits_{p \rightarrow 0} \E_x \Bigl( \ee^{-q \tau_{b,U}^-} \1_{\{\tau_{b,U}^- < \infty\}} \Bigr) &= {Z}^{(q+\lambda)}(x-b) - \frac{q+\lambda}{\Phi_{q+\lambda}} W^{(q+\lambda)}(x-b) \notag \\
			&\quad \quad + Z^{(q+\lambda)}(-b) \cdot \frac{ \lambda \int^\infty_b \underline{r}^{(q+\lambda)}(x,y) \dd y}{\lambda \int^\infty_b \underline{r}^{(q+\lambda)}(0,y) \dd y}, \quad x \in [b,\infty). \notag
		\end{align} 
	\end{itemize}
\end{proposition}

	\subsection{More explicit partial resetting scale functions for  $a,b \in \R_+$} \label{Subsec1:ExitUpwardsInPositiveRealAxis}
	In this subsection, following Remark \ref{ap4}, we consider the case that $a,b \in \R_+$, in which the use of classical scale functions $W^{(q)}$ and $Z^{(q)}$ and their properties  can satisfy the the conditions in Eq.~\eqref{Eq:ConvergenceOfIntegralOperatorForalln} of Theorem \ref{Thm:Res-UAuxiliaryFuncIsFixedPoint}.  The overall goal of this section is thus to derive an analogous theory to that from Sections \ref{sec3.1} and \ref{Subsubsec:Res-ExitUpwardsPostiveHalf} using the classical scale functions rather than the potential densities used in the previous sections.  This will also lead to some simplified expression for exit identities   of Theorems \ref{Thm:Res-TwoSidedExitUpwards-PositiveHalfAxis} and \ref{Thm:Res-TwoSidedExitDownwards-PositiveHalfAxis}. The proofs of the following results are done in exactly the same manner as those in Sections \ref{sec3.1} and \ref{Subsubsec:Res-ExitUpwardsPostiveHalf}, and we will therefore only provide proofs for parts that have non-obvious differences.  
	For the remainder of these sections, we shall assume that $a,b \in \R_+$ and hence have that $C_{[0,a]}$ which was defined previously as the space of continuous functions with domain $[0,a]$.
	
Recall that $W^{(q)}$ the scale function of  a SNLP. We  define for $n \geq 0$ and $q \geq 0$ the associate  functions 
	\[w^{(q,p)}_n (x) = p^n W^{(q)}(p^n x),\]
	and  the convolution of these functions as
	\begin{equation*}
		[\circledast^n_{i=0} w_i^{(q,p)}](x) := [w^{(q,p)}_0 \circledast \dots \circledast w^{(q,p)}_n](x) = \int^x_0 p^n W^{(q)}(p^n(x-y)) 	[\circledast^{n-1}_{i=0} w_i^{(q,p)}](y) \dd y , \quad x \geq 0, \label{Eq:ConvolutionWScaleFunc}
	\end{equation*}
	for which $\circledast$ denotes the standard convolution operator, and where for $n < 0$ we let $[\circledast^n_{i=0} w_i^{(q,p)}](x) = \delta_0(x)$, the standard Dirac-delta function. Furthermore, for $p \in (0,1)$, $q$, $x$, $a$, $u \geq 0$ with $a \geq x$
	and $z,\gamma \in \R$, define for $h \in C_{[0,a]}$ the following operator 
	\begin{align*}
		\mathcal{G}_\gamma^{(q,p)}h(x;u,z) = \sum^{\infty}_{k=0} \gamma^k \int^{x-up^{-k}}_0 h(p^k (x-y)-z)\; \bigl[ \circledast^{k-1}_{i=0} w_i^{(q,p)} \bigr](y) \dd y, \label{Eq:DefnOfGIntegralOperator} 
	\end{align*}
	for which we use the convention that $\mathcal{G}_\gamma ^{(q,\gamma)}h(x;u):=\mathcal{G}_\gamma^{(q,\gamma)}h(x;u,u)$ (it will be seen in the upcoming that we shall usually have the parameters $u=z$). 
	
	We shall now show that $\mathcal{G}_\infty^{(q,\gamma)}h$ is indeed well-defined and convergent, and that it is the unique fixed point of a specified class of integral operators. 
	
	\begin{theorem}\label{Thm:UAuxiliaryFuncIsFixedPoint}
		Let $p \in (0,1)$, $q,x,a,u \geq 0$ with $a \geq x$ and $z,\gamma \in \R$. Furthermore, for $f \in  C_{[0,a]}$, consider the integral operator 
		\begin{equation}
			\mathcal{A}f(x) = h(x-z) + \frac{\gamma}{p} \int^{xp}_{u} W^{(q)}(x - yp^{-1}) f(y) \dd y,  \label{Eq:MainIntegralOperator}
		\end{equation}
		for a chosen $h \in  C_{[0,a]}$, and for which we shall denote $\mathcal{A}^0f(x) = f(x)$ and  $\mathcal{A}^{n+1}f(x) = \mathcal{A}[\mathcal{A}^{n}f(x)]$ for $n \geq 0$.  Then, the unique fixed point of $\mathcal{A}$ is $\mathcal{G}^{(q,p)}_\gamma h(x;u,z)$. 
		%
	\end{theorem}
	
	\begin{proof}
	Let 
		\begin{equation}
			\mathtt{I}_0^{(q)}(x;u,z) := h(x-z), \quad 	\mathtt{I}_n^{(q)}(x;u,z) := \int^{xp}_{u} W^{(q)}(x - yp^{-1}) \mathtt{I}^{(q)}_{n-1}(y;u,z) \dd y,  
		\end{equation} 
		and 
			\begin{align}
			g_n(x) = \mathcal{A}g_{n-1}(x) =  \sum^{n-1}_{k=0} \Bigl(\frac{\gamma}{p} \Bigr)^k \mathtt{I}_k^{(q)}(x;u,z), \quad  n \geq 1. \notag 
		\end{align}
		Similar to Theorem \ref{Thm:Res-UAuxiliaryFuncIsFixedPoint}, one can show that  $g_n(x) =  \mathcal{A}g_{n-1}(x) = \mathcal{A}^{n-1}g_{0}(x)$ converges as $n \rightarrow \infty$  and 
		\begin{align}
			g(x) = \lim\limits_{n\rightarrow \infty}\mathcal{A}g_{n}(x) =   \sum^{\infty}_{k=0} \Bigl(\frac{\gamma}{p} \Bigr)^k \mathtt{I}_k^{(q)}(x;u,z), \label{Eq:FinalSolutionThm1} 
		\end{align}
		is the unique fixed point of $\mathcal{A}$. 
		
	\noindent 	Next, we show that the Laplace transform of $\mathtt{I}_n^{(q)}$ that it can be written in terms of $\bigl[ \circledast^{n-1}_{i=0} w_i^{(q,p)} \bigr](x)$ thus yielding the form of $\mathcal{G}_\gamma^{(q,p)}h$. Hence, denote by $_u\widehat{f}(\vartheta) := \int_{u}^{\infty} \ee^{-\vartheta x} f(x) \dd x$ the (incomplete) Laplace transform of an arbitrary measurable function $f:\R \rightarrow [0,\infty)$. Then, observe for $n \geq 1$ that 
		\begin{align}
			_u\widehat{\mathtt{I}}_n^{(q)}(\vartheta) &= \int^\infty_{u} \ee^{-\vartheta x} \Bigl( \int^{xp}_{u} W^{(q)}(x - yp^{-1}) \mathcal{I}^{(q)}_{n-1}(y;u,z) \dd y \Bigr) \dd x \notag \\ 
			&= \frac{1}{\psi_{q}(\vartheta)} \; _u\widehat{\mathtt{I}}^{(q)}_{n-1}(\vartheta p^{-1}) \notag \\ 
			&=  \Bigl(\prod^{n-1}_{k=0} \frac{1}{\psi_q(\vartheta p^{-k})} \Bigr) \;  _u\widehat{\mathtt{I}}^{(q)}_{0}(\vartheta p^{-n}), \notag 
		\end{align} 
		for which we have that
		\begin{align}
			_u\widehat{\mathtt{I}}^{(q)}_{0}(\vartheta p^{-n}) &= \int^\infty_{u} \ee^{-\vartheta p^{-n} x} h(x-z) \dd x	=p^n \ee^{-\vartheta up^{-n}}  \int^\infty_{0} \ee^{-\vartheta x} h(p^n x +u - z) \dd x. \notag 
		\end{align}
		Furthermore, for $n \geq 1$, we have by using Eq.~\eqref{Eq:ConvolutionWScaleFunc}  that 
		\begin{align}
			\prod^{n-1}_{i=0} \frac{1}{\psi_{q}(\vartheta p^{-i})}
			&= \int^\infty_0 \ee^{-\vartheta x} \; \bigl[ w_0^{(q,p)} \circledast \cdots \circledast w_{n-1}^{(q,p)} \bigr](x) \dd x \notag \\
			&= \int^\infty_0 \ee^{-\vartheta x} \; \bigl[ \circledast^{n-1}_{i=0} w_i^{(q,p)} \bigr](x) \dd x. \label{Eq:LTofConvolutionofwScaleFuncs}
		\end{align}
		Thus, by using the previous two equations, $_u\widehat{\mathtt{I}}_n^{(q)}(\vartheta)$ becomes 
		\begin{align}
			_u\widehat{\mathtt{I}}_n^{(q)}(\vartheta) &=  p^n \ee^{-\vartheta up^{-n}}  \Bigl(\int^\infty_{0} \ee^{-\vartheta x} h(p^nx +u - z) \dd x \Bigr) \cdot \Bigl( \int^\infty_0 \ee^{-\vartheta y} \; \bigl[ \circledast^{n-1}_{i=0} w_i^{(q,p)} \bigr](y) \dd y \Bigr) \notag \\
			&=   \int^\infty_{u} \ee^{-\vartheta x} \Bigl( p^n  \int^{x-up^{-n}}_0 h(p^n(x-y) - z) \; \bigl[ \circledast^{n-1}_{i=0} w_i^{(q,p)} \bigr](y) \dd y \Bigr) \dd x, \notag 
		\end{align}
		where the final step follows by using that $x - up^{-n} < 0$ for $x \in [u,up^{-n})$ and that, for $y<0$,  $\bigl[ \circledast^{n-1}_{i=0} w_i^{(q,p)} \bigr](y) = 0$. Hence, the inverse Laplace transform of the above equation yields for $n \geq 1$ that 
		\begin{equation}
			\mathtt{I}_n^{(q)}(x;u,z) = p^n   \int^{x-up^{-n}}_0 h(p^n(x-y) - z) \; \bigl[ \circledast^{n-1}_{i=0} w_i^{(q,p)} \bigr](y) \dd y, \notag 
		\end{equation}
		for which the above holds also for $n=0$ by recalling that $\bigl[ \circledast^{n-1}_{i=0} w_i^{(q,p)} \bigr](y) = \delta_0(y)$ when $n<1$. The proof is hence completed by substituting the above equation into Eq.~\eqref{Eq:FinalSolutionThm1}. 
	\end{proof}
	
	
\noindent 	For $a$, $b\in\mathbb R_+$, we define the scale functions operators 
$	\mathcal{W}^{(q+\lambda,p)}_\gamma(x;u) :=  \mathcal{G}^{(q+\lambda,p)}_\gamma W^{(q+\lambda)}(x;u)$ and $	\mathcal{Z}^{(q+\lambda,p)}_\gamma(x;u) :=  \mathcal{G}^{(q+\lambda,p)}_\gamma Z^{(q+\lambda)}(x;u)$, given by 
	\begin{align*}\begin{split}	\mathcal{G}_\gamma^{(q,p)}W^{(q)}(x;u) = \sum^{\infty}_{k=0} \gamma^k \int^{x-up^{-k}}_0 W^{(q)}(p^k (x-y)-u)\; \bigl[ \circledast^{k-1}_{i=0} w_i^{(q,p)} \bigr](y) \dd y, \\
\mathcal{G}_\gamma^{(q,p)}Z^{(q)}(x;u) = \sum^{\infty}_{k=0} \gamma^k \int^{x-up^{-k}}_0 Z^{(q)}(p^k (x-y)-u)\; \bigl[ \circledast^{k-1}_{i=0} w_i^{(q,p)} \bigr](y) \dd y,	
 \end{split}\end{align*}
For $a$, $b \in \R_+$	Using  Theorem \ref{Thm:UAuxiliaryFuncIsFixedPoint} and the similar  arguments as in  Theorem \ref{Thm:Res-TwoSidedExitUpwards-PositiveHalfAxis}  and \ref{Thm:Res-TwoSidedExitDownwards-PositiveHalfAxis}, we have the following results.
	
	\begin{theorem}[Upward two-sided exit problem]\label{Thm:TwoSidedExitUpwards-PositiveHalfAxis}
		Let $a,b \in \R_+$,  $x \in [b,a]$ and  $p \in (0,1)$. Then, for $q,\lambda \geq 0$ and $a,b \in \R_+$, it holds 
		\begin{equation}
			\E_x \Bigl( \ee^{-q \tau_{a,U}^+} \1_{\{\tau_{a,U}^+ < \tau_{b,U}^-\}} \Bigr) = \frac{\mathcal{W}_{-\lambda}^{(q+\lambda,p)}(x;b)}{\mathcal{W}_{-\lambda}^{(q+\lambda,p)}(a;b)}. 
		\end{equation}
		\end{theorem}
	
	
	%
	\begin{theorem}[Downward two-sided exit problem]\label{Thm:TwoSidedExitDownwards-PositiveHalfAxis}
		Let $a,b \in \R_+$,  $x \in [b,a]$ and  $p \in (0,1)$. Then, for $q,\lambda \geq 0$ and $a,b \in \R_+$, it holds 
		\begin{align}
			\E_x \bigl( \ee^{-q \tau_{b,U}^-} \1_{\{\tau_{b,U}^- < \tau_{a,U}^+\}} \bigr) 	&=  \mathcal{Z}_{-\lambda}^{(q+\lambda,p)}(x;b) - \lambda \int^{bp^{-1}}_{b} \mathcal{W}_{-\lambda}^{(q+\lambda,p)}(x;b,u) \dd u   \notag \\
			&\quad -  \frac{\mathcal{W}_{-\lambda}^{(q+\lambda,p)}(x;b)}{\mathcal{W}_{-\lambda}^{(q+\lambda,p)}(a;b)} \Bigl( \mathcal{Z}_{-\lambda}^{(q+\lambda,p)}(a;b) - \lambda \int^{bp^{-1}}_{b} \mathcal{W}_{-\lambda}^{(q+\lambda,p)}(a;b,u) \dd u  \Bigr) , \label{Eq:ExitDownwardsParticularZIdentity}
		\end{align}
	\end{theorem}
	\begin{remark}\upshape{
		The additional integral terms appearing in the above identity are analogous to that appearing in the case for the general resolvent identity of Theorem \ref{Thm:Res-TwoSidedExitDownwards-PositiveHalfAxis}  for $a$, $b\in\mathbb R_+$. In the special case for $b = 0$, we have that 
		$$
		\mathcal{Z}_{-\lambda}^{(q+\lambda,p)}(x;0) = 1 + q \int^x_0 \mathcal{W}^{(q+\lambda,p)}_{-\lambda p}(y;0) \dd y,
		$$
		which is analogous to the classical definition of  $Z^{(q)}$. }
	\end{remark}
	
	\section*{Acknowledgement}
The authors are grateful to Takis Konstantopoulos for suggesting the approach considered in Theorem  \ref{sdepart}.	
	
	

	\bibliography{ApostolosV2.bib}
	\bibliographystyle{abbrv}
\end{document}